\documentclass[onefignum,onetabnum]{siamart171218}

\usepackage{amsfonts,amssymb,amsmath,mathrsfs,cite,hyperref,color,graphicx,mathtools}

\title{Stability and convergence of the string method for computing minimum energy paths}

\author{Brian Van Koten\thanks{BvK was supported by the NSF RTG: Computational and Applied Mathematics in Statistical Science, award number 1547396.} \and Mitchell Luskin\thanks{This material is based upon work partly supported by the U.S. Department of Energy Office of Science grant DE-SC0012733.  Part of this research was also performed while the author was a Simons Participant at the NSF supported Institute for Pure and Applied Mathematics (IPAM),
and acknowledges support from the Simons Foundation. }}

\newcommand{\ignore}[1]{}

\newcommand{\Real}{\mathbb{R}}
\newcommand{\hd}{d_{\rm H}}
\newcommand{\dist}{d}

\newcommand{\eps}{\varepsilon}
\DeclareMathOperator{\proj}{proj}
\newcommand{\NN}{\mathscr{N}}

\newcommand{\U}{\mathscr{U}}
\newcommand{\LL}{\mathscr{L}}
\newcommand{\B}{\mathscr{B}}
\newcommand{\grad}{\nabla}

\newcommand{\I}{I}
\newcommand{\R}{R}

\newcommand{\mep}{{\rm MEP}}
\newcommand{\m}{{\rm M}}
\renewcommand{\l}{\ell}
\newcommand{\lstar}{\l^\ast}
\newcommand{\strings}{\mathbf{S}}
\newcommand{\CC}{\mathscr{M}}

\newcommand{\curves}{\mathbf{C}}
\newcommand{\tmin}{\tau}
\newcommand{\tmax}{t_{\rm{max}}}

\newcommand{\ns}{\bar{S}}

\newcommand{\C}{{\rm C}}

\newcommand{\lya}{W}

\newsiamthm{assumption}{Assumption}
\newsiamremark{remark}{Remark}

\begin{document}
\maketitle

\begin{abstract}
  We analyze the convergence of the string method of E, Ren, and Vanden-Eijnden~\cite{ERenVandenEijnden:SimplifiedString2007} to a minimum energy path.
  Under some assumptions relating to the critical points on the minimum energy path, we show that the string method initialized in a neighborhood of the minimum energy path converges to an arbitrarily small neighborhood of the minimum energy path as the number of images is increased.
\end{abstract}

\section{Introduction}

Many systems in chemistry, materials science, and physics pass through sequences of metastable states, undergoing transitions between states very rarely. Often the best way to understand these systems is to catalog the states together with the rates and most probable mechanisms of the transitions.
However, when transitions are rare events, the cost of calculating rates and mechanisms by direct simulation may be prohibitive.
A less costly alternative is to compute Minimum Energy Paths ($\mep$s) connecting metastable states. Under some conditions, one may interpret a $\mep$ as a representative transition mechanism~\cite[Section~1.3]{cameron_string_2011}, and given a $\mep$ one may estimate transition rates using approximations such as harmonic transition state theory~\cite{hanggi_reaction_rate_1990,vineyard_frequency_1957,berglund_kramers_2011}.
Several algorithms exist for computing $\mep$s, including the nudged elastic band method~\cite{HenkJonss:ImprovedTangentNEB2000} and the string method~\cite{ERenVandenEijn:String2002,ERenVandenEijnden:SimplifiedString2007}; see~\cite{sheppard_optimization_2008,henkelman_methods_2002} for surveys.
We analyze the simplified and improved string method~\cite{ERenVandenEijnden:SimplifiedString2007}, giving the first proof of convergence of a practical algorithm for computing $\mep$s.

A $\mep$ is a path connecting two local minima of the potential energy $V$ whose tangent is everywhere parallel to $\grad V$, except at critical points where $\grad V =0$; cf.\@ equation~\eqref{eq: characterization of mep}. In the simplest case, a $\mep$ follows a steepest ascent trajectory from one local minimum up to a saddle point and then a steepest descent trajectory down to the other local minimum; see Figure~\ref{fig: mep notation}. Each local minimum corresponds to a metastable state, and a $\mep$ yields information about transitions between the states which it connects. In particular, harmonic transition state theory provides approximations expressing transition rates in terms of $V$ and $D^2 V$ evaluated at saddle points and local minima~\cite{vineyard_frequency_1957,berglund_kramers_2011,hanggi_reaction_rate_1990}. The relevant saddle for a given transition between metastable states is the highest energy point on a $\mep$ connecting those states.

$\mep$s also yield insight into transition mechanisms.
Suppose that the system evolves under the overdamped Langevin dynamics
\begin{equation*}
  dx_t = -\grad V(x_t) dt + \sqrt{2 \beta^{-1}} dB_t
\end{equation*}
with inverse temperature $\beta$.
This is perhaps the simplest model of a metastable system where thermal fluctuations drive transitions.
If $V$ has exactly two minima and one saddle, then the unique $\mep$ minimizes the Wentzell--Friedlin action, so trajectories passing between neighborhoods of the minima typically lie close to the $\mep$ when $\beta$ is large~\cite[Section~1.3]{cameron_string_2011}. That is, the $\mep$ is representative of typical transitions between the metastable states.

Many algorithms have been proposed to compute $\mep$s~\cite{ERenVandenEijnden:SimplifiedString2007,HenkJonss:ImprovedTangentNEB2000,ERenVandenEijn:String2002,fischer_conjugate_1992,ulitsky_new_1990}. Several of these, including the string method~\cite{ERenVandenEijn:String2002,ERenVandenEijnden:SimplifiedString2007} and the nudged elastic band method~\cite{HenkJonss:ImprovedTangentNEB2000}, discretize the gradient descent dynamics on curves (GDDC). Under GDDC, each point of a curve evolves independently by gradient descent for $V$, cf.\@ equation~\eqref{eq: gradient descent on curves}. Any $\mep$ is a fixed point of GDDC. This is because $\mep$s consist of heteroclinic paths and stationary points of the gradient descent, and these sets are always invariant. Moreover, under some conditions, one can show that trajectories of GDDC converge to $\mep$s; see~\cite[Corollary~4]{cameron_string_2011} and Theorem~\ref{thm: uniform and asymptotic stabilit of mep} below. These observations suggest that to calculate $\mep$s, one might try to simulate GDDC until convergence.  However, the results of~\cite{cameron_string_2011} do not imply convergence of any discretization of GDDC.

We present the first convergence analysis of a practical, discretized method for calculating $\mep$s. Our analysis assumes that the $\mep$ passes alternately through local minima and saddles of index one as depicted in Figure~\ref{fig: mep notation} and that no degenerate critical points or critical points of higher index lie on the $\mep$.~\footnote{A saddle point $s$ of index $m$ is a critical point of the potential energy $V$ where the Hessian $D^2 V(s)$ is invertible and has exactly $m$ negative eigenvalues. A degenerate critical point is one where the Hessian is singular.}
 Under these conditions, we show in Theorem~\ref{thm: convergence} that the $\mep$ can be calculated to arbitrary precision using the simplified and improved string method~\cite{ERenVandenEijnden:SimplifiedString2007}. We allow only minima and  index one saddle points along the $\mep$, since when saddles of index two or higher are present, GDDC need not converge to a path~\cite{cameron_string_2011}. In that case, we do not expect discretizations of GDDC to converge either. Moreover, in most applications only saddles of index one are relevant. We discuss these issues at length in Remark~\ref{rem: justification of index one assumption}.

Our convergence analysis relies on new stability results for GDDC. In particular, we show that any $\mep$ alternating between minima and saddles of index one is uniformly and asymptotically stable under GDDC; cf.\@ Theorem~\ref{thm: uniform and asymptotic stabilit of mep}. Previous work was concerned primarily with $\omega$-limit sets of trajectories of GDDC or with global convergence to a $\mep$~\cite{cameron_string_2011}; cf.\@ Remark~\ref{rem: comparison with cameron article}. As far as we are aware, Theorem~\ref{thm: uniform and asymptotic stabilit of mep} is the first result on local stability of individual $\mep$s.

We do not address variations of the string method combining gradient descent dynamics with sampling, such as the string method in collective variables~\cite{maragliano_string_2006}, the on-the-fly string method~\cite{maragliano_--fly_2007}, or the finite temperature string method~\cite{e_finite_2005}. Neither do we address the nudged elastic band method~\cite{HenkJonss:ImprovedTangentNEB2000} or methods combining gradient descent dynamics with minimum mode following such as the climbing image nudged elastic band method~\cite{henkelman_climbing_2000}. However, we hope that our results will facilitate the analysis of these and similar methods in future work.

\section{Minimum Energy Paths and the String Method}\label{sec: algorithm}

The string method is a numerical algorithm for computing minimum energy paths.
Let $\Real^d$ be the \emph{state space} of the system. Let $V: \Real^d \rightarrow \Real$ be a \emph{potential energy}.
Let $m,m' \in \Real^d$ be local minima of $V$. A \emph{minimum energy path} ($\mep$) connecting $m$ to $m'$ is a path $\phi: [a,b] \rightarrow \Real^d$ so that $\phi(a) =m$, $\phi(b) =m'$, and
\begin{equation}\label{eq: characterization of mep}
  \grad V(\phi(\alpha)) - \left \langle \grad V(\phi(\alpha) ),\frac{\phi'(\alpha)}{\lVert \phi'(\alpha) \rVert} \right \rangle \frac{\phi'(\alpha)}{\lVert \phi'(\alpha) \rVert} = 0
\end{equation}
for all $\alpha \in [a,b]$ except those $\alpha$ so that $\phi(\alpha)$ is a critical point of $V$.
We allow kinks in a $\mep$ at critical points, so $\phi'(\alpha)$ need not exist when $\phi(\alpha)$ is a critical point, and in that case the left hand side of~\eqref{eq: characterization of mep} is undefined at $\alpha$.
Equation~\eqref{eq: characterization of mep} implies that either ${\grad V(\phi(\alpha))=0}$ or ${\grad V(\phi(\alpha)) \parallel \phi'(\alpha)}$. Therefore, $\mep$s are composed of heteroclinics of the gradient descent dynamics
\begin{equation}\label{eq: gradient descent ode}
  \dot x = -\grad V(x),
\end{equation}
connecting $m$ to $m'$ through a sequence of intermediate critical points of $V$; see Figure~\ref{fig: mep notation}.
In our analysis, we generally prefer to treat $\mep$s as curves instead of paths, even though they are called paths. A curve $\Gamma$ is a set of the form $\Gamma=\gamma([a,b])$ parametrized by some continuous path $\gamma:[a,b] \rightarrow \Real^d$. Each curve admits parametrization by many different paths.

\begin{figure}
  \includegraphics{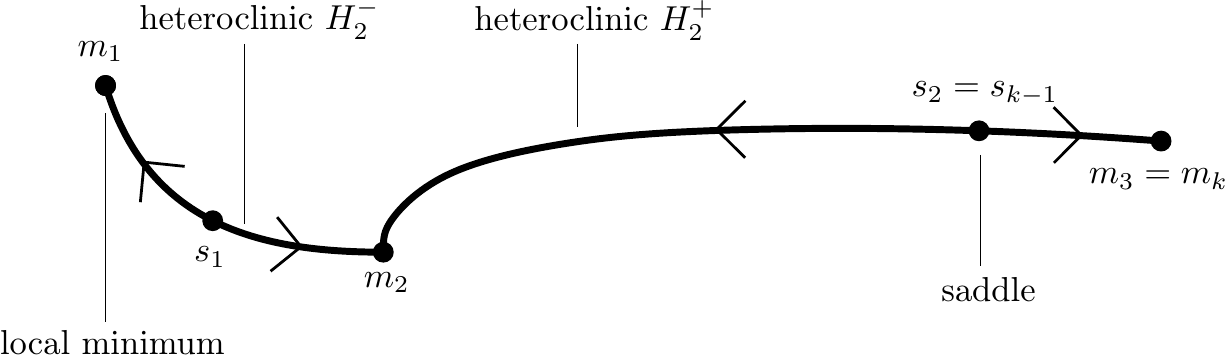}
  \caption{A minimum energy path ($\mep$) consists of a sequence of heteroclinic trajectories of the gradient flow connecting critical points. We assume in our analysis that the $\mep$ passes alternately through local minima $m_1, m_2, \dots, m_k$ and saddle points $s_1, s_2, \dots, s_{k-1}$, as pictured. However, in general, $\mep$s may contain maxima, higher index saddles, or degenerate critical points.
Moreover, saddles of index one may be connected directly by heteroclinics without a minimum in between as in Figure~\ref{fig: meps connecting saddles directly}. We denote the heteroclinics connecting $m_i$ with $s_{i-1}$ and $s_i$ by $H_i^-$ and $H_i^+$, respectively. The arrow on each heteroclinic points in the gradient descent direction.}
  \label{fig: mep notation}
\end{figure}

Several methods for finding $\mep$s, including the string method~\cite{ERenVandenEijn:String2002,ERenVandenEijnden:SimplifiedString2007} and nudged elastic band method~\cite{HenkJonss:ImprovedTangentNEB2000}, are based on discretizations of the gradient descent dynamics on curves (GDDC).
Under GDDC, every point of a curve evolves independently by gradient descent. To be precise, let $S_t: \Real^d \rightarrow \Real^d$ be the flow for~\eqref{eq: gradient descent ode}. That is, for $t\geq 0$ and $x \in \Real^d$, define $S_t(x) := z(t)$ for $z$ the solution of the initial value problem
\begin{align*}
  z' = -\grad V(z) \text{ and } z(0) = x.
\end{align*}
Let $\Gamma_0$ be a curve. Under GDDC, $\Gamma_0$ evolves by
\begin{equation}\label{eq: gradient descent on curves}
  \Gamma_t = S_t(\Gamma_0) \text{ for } t \geq 0.
\end{equation}

Any $\mep$ is a fixed point of GDDC. This is because $\mep$s consist of stationary points and heteroclinics, and these sets are always invariant under the flow. (However, no parametrization $\phi$ of a $\mep$ is invariant under the dynamics $\phi_t = S_t \circ \phi$ corresponding to GDDC. As $t$ increases, points on the heteroclinics fall towards the minima, so $\phi_t$ spends more and more time near minima, undergoing increasingly abrupt transitions through the saddle points.) Moreover, under some conditions, one can show that trajectories of GDDC converge to $\mep$s; cf.\@ Theorem~\ref{thm: uniform and asymptotic stabilit of mep} and~\cite{cameron_string_2011}. Thus, to calculate $\mep$s, one might try to simulate the GDDC until convergence. Both the string method and nudged elastic band method adopt this strategy.


We now introduce one particular variant of the string method, called the simplified and improved string method~\cite{ERenVandenEijnden:SimplifiedString2007}. Our first task is to define a finite dimensional family of paths connecting the local minima $m$ and $m'$. For $N \in \mathbb{N}$, we let
\begin{equation*}
  \strings^{N+1} := \left \{x \in  (\Real^d)^{N+1}: x_0 = m, x_{N} = m'\right\}.
\end{equation*}
We call components $x_i$ of $x \in \strings^{N+1}$ \emph{images},
and we call $x \in \strings^{N+1}$ a \emph{string of $N+1$ images connecting $m$ with $m'$}.
By convention, we index all strings of $N+1$ images by $0,1, \dots, N$.
For $x \in \strings^{N+1}$, we define the \emph{length} $\l(x) \in \Real^{N+1}$
and \emph{normalized length} $\lstar(x) \in \Real^{N+1}$ of $x$ by
\begin{equation*}
  \l(x)_i = \sum_{k = 1}^i |x_k - x_{k-1}|\text{ and } \lstar (x)_i := \frac{\l(x)_i}{\l(x)_N},
\end{equation*}
respectively.
We define
\begin{equation*}
  m(x):= \max_{i = 1, \dots, N+1} |x_i - x_{i-1}|
\end{equation*}
to be the \emph{string spacing}.

An \emph{interpolant} is an operator $\I$ which interpolates the images of a string to form a continuous path from $m$ to $m'$.
That is, for $\alpha \in \Real^{N+1}$ with $\alpha_i > \alpha_{i-1}$
and $x \in \strings^{N+1}$, $I(\alpha, x):[\alpha_0, \alpha_N]\rightarrow \Real^d$
is a continuous path with
\begin{equation*}
  \I(\alpha,x)(\alpha_i) = x_i
  \text{ for all } i = 0, 1, \dots, N.
\end{equation*}
Standard examples include the linear and cubic spline interpolants.
The linear interpolant is defined piecewise by the formula
\begin{equation}\label{eq: linear interpolant}
  \I(\alpha,x)(\beta) = \frac{\alpha_{i+1} - \beta}{\alpha_{i+1}-\alpha_i} x_i +\frac{\beta-\alpha_i}{\alpha_{i+1}-\alpha_i} x_{i+1} \text{ for } \beta \in [\alpha_i, \alpha_{i+1}].
\end{equation}
For $x \in \strings^{N+1}$, we let $\I x$ denote the curve parametrized by the path $I(\lstar(x), x)$.

The novel difficulty in discretizing GDDC is that when the images $x_i$ evolve by gradient descent, they fall away from the saddle, converging to local minima. To counteract this effect, one must periodically equalize the spacing between images. We call this procedure reparametrization of the string. The \emph{reparametrization operator} used in the simplified and improved string method is an easily computed analogue of parametrization by arc length for strings. Given $y \in \strings^{N+1}$ and an interpolant $I$, we define $\R:\strings^{N+1} \rightarrow \strings^{N+1}$  by
\begin{displaymath}
  R(y)_i := I(\lstar(y), y)\left( \frac{i}{N} \right) \mbox{ for } i = 0, \dots, N.
\end{displaymath}
We show in Lemma~\ref{lem: interp error} that if $I$ is the linear interpolant, then as long as $\frac{\l(y)_N}{N} \leq h$, the reparametrized string has $m(R(y)) \leq h$. Thus, the reparametrization operator moves the images close together if they are initially far apart.

Finally, to evolve strings by gradient descent, we require a numerical integrator for the gradient flow. We let $\ns_t$ denote this integrator. One might take $\ns_t$ to be a Euler's method, for example.

We will analyze the following version of the string method:
Pick $h>0$, $K>1$, and $\Delta t >0$.
As the notation suggests,
$h$ and $\Delta t$ are the spatial and temporal discretization parameters of the string method.
Let $x^0 \in \strings^{N+1}$ be an initial guess for a string following the minimum energy path.
Assume that
\begin{displaymath}
  m(x^0)= \leq h.
\end{displaymath}
The simplified and improved string method
iterates advancing the string by the numerical gradient flow $\ns_{\Delta t}$
and reparametrizing the string when the distance between images exceeds $Kh$.
To be precise,
suppose that for some $n \in \mathbb{N}$, a string $x^n$ has been computed.
To compute $x^{n+1}$, first let $y \in \strings^{N+1}$ with
\begin{displaymath} y_i :=  \ns_{\Delta t} (x^n_i).\end{displaymath}
Next, check whether $m(y) \leq Kh$.
If so, let
\begin{displaymath} x^{n+1} := y. \end{displaymath}
If not, then check whether $\frac{\l(y)_N}{N} \leq h$.
If $\frac{\l(y)_N}{N} \leq h$, let
\begin{displaymath}
  x^{n+1} := \R (y).
\end{displaymath}
As explained above, if $I$ is the linear interpolant,
then as long as $\frac{\l(y)_N}{N} \leq h$,
the reinterpolated string has $m(x^{n+1}) \leq h$.
If $\frac{\l(y)_N}{N} > h$, then more images must be added to the string to guarantee $m(x^{n+1}) \leq h$:
Let $N' := \left  \lceil \frac{\l(y)_N}{h} \right \rceil$ and define $x^{n+1} \in (\Real^d)^{N'+1}$ by
\begin{displaymath}
  x^{n+1}_i := I(\lstar(y), y)\left( \frac{i}{N'}\right) \mbox{ for } i = 0, \dots, N'.
\end{displaymath}
(For efficiency, it may also be desirable to reduce the number of points
on the string if $\frac{\l(x^n)_N}{N}$ falls below a certain threshold, but we do not consider this possibility.)

\section{Local Stability of Minimum Energy Paths}
Our goal in this section is to analyze the local stability of minimum energy paths under the GDDC,  providing a basis for a numerical analysis of the string method.
Throughout the rest of this work, we restrict our attention to a single minimum energy path, denoted $\m$, and satisfying the following conditions:

\begin{assumption}\label{asm: assumptions on mep}
  $\m$ is a minimum energy path passing through the local minima $m_1, \dots, m_k$ and the saddles $s_1, \dots, s_{k-1}$, as illustrated in Figure~\ref{fig: mep notation}.
The critical points $m_1, \dots, m_k$ and  $s_1, \dots, s_{k-1}$ are distinct and isolated. No other critical points lie on $\m$.
 Each of the saddles has Morse index one, i.e.\@ for all ${i=1, \dots, k-1}$, $D^2 V (s_i)$ has one negative eigenvalue and all its other eigenvalues are positive. Each local minimum is linearly stable, i.e.\@ $D^2V(m_i)$ is positive definite.
No two of the saddles $s_1, \dots, s_{k-1}$ are connected by a single heteroclinic along $\m$ without a minimum in between. That is, $\m$ passes through minima and saddles alternately.
$\m$ does not intersect itself, so it is homeomorphic to an interval.
\end{assumption}

\begin{remark}
    \label{rem: justification of index one assumption}
    We restrict our attention to $\mep$s passing through critical points of index less than or equal to one for two reasons:
  First, the GDDC may not converge to a $\mep$ if the initial curve lies near a saddle of index two or higher~\cite{cameron_string_2011}. In that case, we cannot expect the string method to converge either.
  Second, in most applications, only saddles of index one are relevant. Typically, one is most interested in the lowest energy saddle separating two minima $m$ and $m'$. To be precise, define  the \emph{communication height} between $m$ and $m'$ to be
  \begin{equation*}
    H(m,m') = \min_{\phi \in \C(m, m')} \max_{z \in \phi} V(z),
  \end{equation*}
  where $\C(m,m')$ is the set of continuous curves connecting $m$ with $m'$.
  In general, the maximizer $z'$ of the energy along any minimizing curve is a critical point of $V$. In fact, if the Hessian $D^2 V$ is continuous and invertible at $z'$, then $z'$ is a saddle of index one~\cite{murrell_symmetries_1968}. (There do exist molecular systems for which the communication height is attained at a critical point $z'$ where $D^2 V(z')$ is degenerate or where $V$ is not twice continuously differentiable~\cite{wales_limitations_nodate,wales_basins_1992}. However, these systems seem to be rare in practice.)
  Given the lowest energy saddle $z'$, one can estimate the rate of transitions from $m$ to $m'$ using transition state theory~\cite{hanggi_reaction_rate_1990, berglund_kramers_2011,vineyard_frequency_1957}.

  Although we assume that $\m$ alternates between minima and saddles, it is possible for a single heteroclinic to connect two saddles of index one with no local minimum in between. This occurs when the unstable manifold of the higher energy saddle under gradient descent lies within the stable manifold of the lower energy saddle.
  See~\cite{wales_energy_2003} for examples of molecular systems with $\mep$s of this type.
  These $\mep$s are not in general stable under GDDC.
  For example, consider the $\mep$ depicted in solid black in Figure~\ref{fig: meps connecting saddles directly}. The limit sets under GDDC of small perturbations of this curve may include both the solid black and the dashed heteroclinics.
\end{remark}

\begin{figure}
  \includegraphics{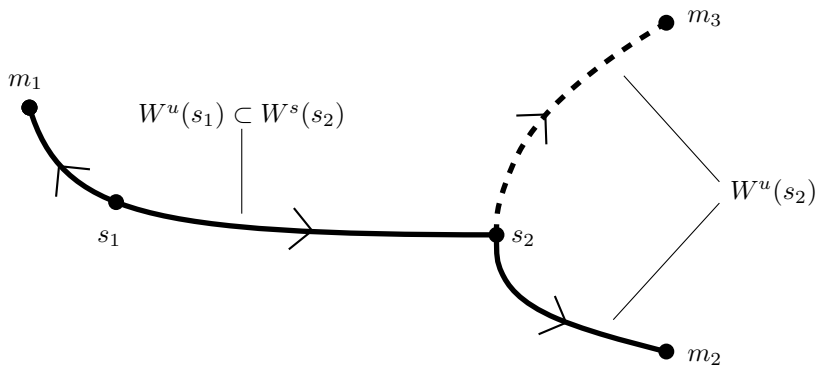}
  \caption{The solid black curve depicts a $\mep$ which connects two saddles of index one directly without passing through a local minimum in between. Here, the unstable manifold $W^u(s_1)$ at the saddle $s_1$ is contained within the stable manifold $W^s(s_2)$ at $s_2$. The $\mep$ might follow either of the two heteroclinics lying on $W^u(s_2)$, leading to the two minima $m_2$ and $m_3$. The arrows on each heteroclinic point in the gradient descent direction.}
  \label{fig: meps connecting saddles directly}
\end{figure}

In addition to our assumptions on the $\mep$ $\m$, we impose the following assumptions on $V$:

\begin{assumption}\label{asm: assumptions on V}
  The potential  $V: \Real^d \rightarrow \Real$ is three times continuously differentiable. The gradient $\grad V$ is globally Lipschitz with constant $L>0$ in the Euclidean norm; that is,
  \begin{equation*}
    \lvert \grad V(x)- \grad V(y) \rvert \leq L \lvert x - y\rvert,
  \end{equation*}
  where $\lvert \cdot \rvert$ denotes the Euclidean norm.
  The third derivatives of $V$ are bounded, i.e., there exists $C_3 >0$ so that
  \begin{equation}
    \sup_{x \in \Real^d} \left  \lvert \frac{\partial^3 V}{\partial x_i \partial x_j \partial x_k}(x) \right \rvert \leq C_3
  \end{equation}
  for all $i,j,k=1, \dots, d$.
\end{assumption}

Before we can discuss stability of GDDC, we require a metric to measure discrepancies between curves. We choose the Hausdorff distance:

\begin{definition}
  For compact sets $X,Y \subset \Real^d$, let
  \begin{displaymath}
    d(X,Y) := \sup_{x \in X} \inf_{y\in Y} |x-y|
  \end{displaymath}
  be the (one-sided) distance from $X$ to $Y$.
  Now define the Hausdorff distance $\hd(X,Y)$ between sets $X$ and $Y$ to be
  \begin{equation*}
    \hd(X,Y) = \max\{d(X,Y), d(Y,X)\}.
  \end{equation*}
\end{definition}

Our proof of convergence relies on the construction of a Lyapunov function for $\m$ in Hausdorff distance. We give a precise definition of Lyapunov function below, following the definitions given in~\cite{Yoshizawa:Stability1966,StuartHumphries:DynamicalSysNumerAn1996} for dynamical systems on $\Real^d$.

\begin{definition}
  Let $B \subset \Real^d$ be a set containing $\m$,
  and let $\curves(m_1,m_k,B)$ denote the set of all continuous curves connecting $m_1$ and $m_2$ and contained in $B$.
  We say that $\lya$ is a
  \emph{Lyapunov function for $\m$} on $B$
  if $\lya: \mathscr{W} \rightarrow [0,\infty)$ for some forward invariant set
  $\mathscr{W} \supset \curves(m_1,m_k,B)$ and the following statements hold:
  \begin{enumerate}
  \item There exists a constant $c>0$ such that for any $\phi \in \mathscr{W}$,
    \begin{equation}\label{eq: contraction property of Lyapunov function}
      \lya(S_t \phi) \leq e^{-ct} \lya(\phi) \mbox{ for all } t\geq 0.
    \end{equation}
  \item For any $\phi, \psi \in \mathscr{W}$,
    \begin{equation}\label{eq: Lipschitz property of Lyapunov function}
      |\lya(\phi) - \lya(\psi)| \leq \hd(\phi, \psi).
    \end{equation}
  \item There exists a strictly increasing, continuous function
    $\alpha: [0,\infty) \rightarrow [0,\infty)$
    with $\alpha(0) = 0$ such that,
    \begin{equation}\label{eq: equivalence with distance of Lyapunov function}
      \alpha(\hd(\m, \phi)) \leq \lya(\phi) \leq \hd(\m, \phi).
    \end{equation}
  \end{enumerate}
\end{definition}

By~\cite[Theorem~2.7.6]{StuartHumphries:DynamicalSysNumerAn1996}, an equilibrium point of a dynamical system on $\Real^d$ has a Lyapunov function if and only if it is uniformly and asymptotically stable.
(Such results are useful in control theory, where they are called converse Lyapunov theorems.)
We will generalize this result to prove the existence of a Lyapunov function for $\m$. First, we give the appropriate definitions of uniform and asymptotic stability:

\begin{definition}[Uniform stability]
  We say that  $\m$ is \emph{uniformly stable} in Hausdorff distance
  if and only if for every ${\eps > 0}$, there exists a $\delta > 0$ so that
  for $\phi \in \curves(m_1, m_k)$,
  ${\hd(\m, \phi) \leq \delta}$ implies $\hd(\m, S_t \phi) \leq \eps$
  for all $t \geq 0$.
\end{definition}

\begin{definition}[Asymptotic stability with uniform convergence]
  Let $B$ be a compact set with $\m \subset B \subset \Real^d$,
  and let $\phi \in \curves(m_1,m_k,B)$.
  We say that $\m$ is \emph{asymptotically stable on $B$ with uniform convergence}
  if for any $\eps >0$, there exists a time $T(\eps, B)$ independent of $\phi$
  so that $\hd(\m, S_t \phi) < \eps$ for all $t > T(\eps, B)$.
\end{definition}

In general, asymptotic stability with uniform convergence is stronger
than asymptotic stability.
Usually, one defines a point $x \in \Real^m$ to be asymptotically
stable on a set $U \subset \Real^m$ under a flow $R_t:\Real^m \times \Real \rightarrow \Real^m$
if and only if it is uniformly stable and
$y \in U$ implies $\lim_{t \rightarrow \infty} R_t y =x$.
Our definition additionally requires that the convergence of $R_t y$ to $x$
be uniform over all $y \in U$.
When the set $U$ is compact, asymptotic stability with
uniform convergence follows from asymptotic stability.
But in our case, the set $\curves(m_1,m_k,B)$ is not compact in the Hausdorff distance
even if $B \subset \Real^d$ is compact.

Our proofs of asymptotic and uniform stability require the following lemma relating the Hausdorff distance between a curve and $\m$ to a one-sided distance:

\begin{lemma}\label{lem: bound on distance gives bound on haus distance}
  For every $\eps >0$ there exists an $\eta >0$ such that
  if $\phi \in \curves(m_1, m_k)$ with $d(\phi, \m) < \eta$,
  then $d(\m,\phi) <\eps$, hence $\hd(\phi , \m) < \eps$.
\end{lemma}

\begin{proof}
  We present only a sketch of the proof here. The details appear in Appendix~\ref{apx: proof of one-sided distance lemma}.
  The existence of a tubular neighborhood of $\m$ would suffice: Given a tubular neighborhood $\mathscr{U}$ of radius $\eps$, any path $\gamma \in \curves(m_1,m_k, \mathscr{U})$ would pass through every normal disk in the neighborhood by the intermediate value theorem. Thus, for any $x$ on $\m$, there would be a point $y$ on $\gamma$ lying in the normal disk centered at $x$ with $\lVert y-x\rVert \leq \eps$. This would prove the result, since for $\eta$ small enough, $d(\phi, \m) \leq \eta$ would imply $\phi \subset \mathscr{U}$.
  Unfortunately, $\m$ may not have a tubular neighborhood, since $\m$ may not be differentiable at the local minima $m_i$.
  In Appendix~\ref{apx: proof of one-sided distance lemma}, we develop an argument based on an analogue of a tubular neighborhood consisting of balls surrounding the minima connected by tubes where $\m$ must be smooth.
\end{proof}

By Lemma~\ref{lem: bound on distance gives bound on haus distance}, to prove uniform and asymptotic stability in Hausdorff distance under GDDC, it suffices to prove the analogous properties for $\m$ under the gradient flow $S_t$ in a one-sided distance; see Lemma~\ref{cor: uniform stability} and Lemma~\ref{cor: asymptotic stability of gradient descent}. Lemma~\ref{lem: construction of stable neighborhoods} is our first step in proving these stability results.

\begin{lemma}\label{lem: construction of stable neighborhoods}
  For any point $p \in \m$ and any $\eps >0$, there exists an open set $\U_\eps(p)$ containing $p$ so that $x \in \U_\eps(p)$ implies $d(S_t(x),\m) \leq \eps$ for all $t\geq 0$.
\end{lemma}
\begin{proof}
  We distinguish three cases: $p$ may be a local minimum, it may lie on a heteroclinic, or it may be a saddle.

  Suppose that $p$ is a local minimum. Since we assume that each local minimum is linearly stable, there exists some $0<\delta\leq\eps$ so that $x \in B_\delta(p)$ implies $S_t(x) \in B_\delta(p)$ for all $t\geq 0$. Thus, we may take $\U_\eps(p) = B_\delta(p)$.

  Now suppose that $p$ lies on the heteroclinic connecting the minimum $m$ with the saddle $s$.
  Let $0< \delta \leq \eps$ be small enough that $x \in B_\delta(m)$ implies $S_t(x) \in B_\delta(m)$ for all $t\geq 0$.
  We have
  \begin{equation*}
    \lim_{t \rightarrow \infty} S_t(p) = m,
  \end{equation*}
  so
  \begin{equation*}
    t(p) := \inf \{t\geq 0 : S_t(p) \in B_{\delta/2}(m)\}
  \end{equation*}
  is finite. Recall that we assume $-\grad V$ is Lipschitz with constant $L$, and define
  \begin{equation}\label{eq: delta on heteroclinic}
    \delta(p) := \frac{\delta}{2 \exp(t(p)L)}.
  \end{equation}
  If $x \in B_{\delta(p)}(p)$, then for all $t \leq t(p)$,
  \begin{displaymath}
    \lvert S_t (p) - S_t(x)\rvert \leq \exp(L t) \delta(p) \leq \exp(L t(p)) \delta (p) \leq \frac{\delta}{2} < \eps.
  \end{displaymath}
  Therefore, since $S_t(p) \in \m$,  $d(S_t(x), \m)< \eps$ when $0 \leq t \leq t(p)$.
  Moreover,
  \begin{equation*}
    \lvert S_{t(p)}(x) -m \rvert \leq \lvert S_{t(p)}(x) - S_{t(p)}(p) \rvert  + \lvert S_{t(p)}(p) - m \rvert \leq \delta,
  \end{equation*}
  so $S_t(x) \in B_\delta(m)$ for all $t \geq t(p)$, hence $d(S_t(x), \m) \leq \eps$ for $t \geq t(p)$.
  Thus, we may take $\U_\eps(p) = B_{\delta(p)}(p)$ when $p$ lies on a heteroclinic.

  Finally, suppose that $p=s_i$ is the saddle point lying between the local minima $m_i$ and $m_{i+1}$. Let $\mathscr{L}_t$ be the flow for $\dot v = -D^2 V (p)v$, and let $v_1$ be the eigenvector corresponding to the unique negative eigenvalue of $D^2V(p)$. By the Hartman--Grobman theorem, there exists an open ball $B_r(0) \subset \Real^d$ and a homeomorphism $\Psi: B_r(0) \rightarrow \Real^d$ so that
  \begin{displaymath}
    S_t \circ \Psi(y) = \Psi \circ \LL_t (y)
  \end{displaymath}
  for all $y \in B_r(0)$ and $t\geq 0$ so that $\LL_t (y) \in B_r(0)$. See Figure~\ref{fig: hartman grobman} for an illustration of the mapping $\Psi$ and the relations between the various neighborhoods defined below. Choose $0<\rho < r$. Let $\B$ be the open cube inscribed in $\partial B_\rho (0)$ with edges parallel to the eigenvectors of $D^2V(p)$. We observe that $\Psi(\Real v_1 \cap \B)$ is a subset of $\m$. In addition, $\Real v_1 \cap \partial \B = \{x_1, x_2\}$, where $y_1 := \Psi(x_1)$ lies on the heteroclinic connecting $p$ with $m_i$ and $y_2 := \Psi(x_2)$ lies on the heteroclinic connecting $p$ with $m_{i+1}$. Since $\Psi$ is uniformly continuous on $\B$, there exists $\gamma >0$ such that for all $x, y \in \B$,
  \begin{equation}\label{eq: modulus of continuity of psi}
    \lvert x-y \rvert <\gamma \text{ implies } \lvert \Psi(x) - \Psi(y)\rvert < \min\{\delta(y_1), \delta(y_2)\},
  \end{equation}
  (Here, $\delta(y_1)$ and $\delta(y_2)$ are defined as in~\eqref{eq: delta on heteroclinic} above.)
  Now let
  \begin{displaymath}
    \CC :=  \{ x \in \B : \lvert x - \proj_{v_1} x \rvert < \gamma \}.
  \end{displaymath}

  We claim that
  \begin{displaymath}
    \U_\eps(p) : = \Psi(\CC)
  \end{displaymath}
  has the desired property.
  To see this, observe that $\lvert \LL_t x - \proj_{v_1} \LL_t x \rvert$ decreases with $t$, since $\Real v_1$ is the unstable subspace of $-D^2 V(p)$. In particular, if $x \in \CC$, then
  \begin{equation}\label{eq: upper bound on projection to unstable space}
    |\LL_t x - \proj_{v_1} \LL_t x| < \gamma
  \end{equation}
  for all ${t \geq 0}$.
  Now let $y \in \U_\eps(p)$, and write $x = \Psi^{-1}(y)$.
  For any $t \geq 0$ such that $S_t(y) \in \U_\eps(p)$, we have
  \begin{align*}
    d(S_t(y), \m) &\leq \lvert S_t(y) - S_t(\Psi(\proj_{v_1}(x))) \rvert \\
                  &= \lvert \Psi( \LL_t(x)) - \Psi (\LL_t (\proj_{v_1}(x))) \rvert \\
                  &\leq \min\{\delta(y_1),\delta(y_2)\}\\
                  &\leq \eps.
  \end{align*}
  The first inequality follows since $\Psi (\proj_{v_1} x) \in \m$, and the second to last follows from~\eqref{eq: modulus of continuity of psi} and~\eqref{eq: upper bound on projection to unstable space}.

  We must now show that $d(S_t(y), \m) \leq \eps$ for all $t \geq 0$, not merely $t$ such that $S_t(y) \in \U_\eps(p)$. If $y$ lies on the stable manifold $W^s (p)$ of $p$, then $S_t(y)  \in \U_\eps(p)$ for all $t \geq 0$, so suppose that ${y \notin W^s(p)}$.
  In that case, we claim that the trajectory $S_t(y)$ passes through either $B_{\delta (y_1)} (y_1)$ or $B_{\delta (y_2)} (y_2)$ as it exits $\U_\eps(p)$. To see this, we observe that by~\eqref{eq: upper bound on projection to unstable space}, the trajectory $\LL_t x$ passes through either $ B_\gamma(x_1)$ or $B_\gamma(x_2)$ as it leaves $\CC$. That is, for $T \geq 0$ the unique time such that $\LL_T x \in \partial \CC$, we have $\LL_T x \in B_\gamma(x_1) \cup B_\gamma(x_2)$.
  Therefore, by~\eqref{eq: modulus of continuity of psi},
  \begin{displaymath}
    S_T(y) = \Psi(\LL_T x) \in B_{\delta (y_1)} (y_1) \cup B_{\delta (y_2)} (y_2).
  \end{displaymath}
  It follows that $d(S_t(y), \m) \leq \eps$ for all $t \geq T$.
  Thus, in fact, ${d (S_t(y), \m) \leq \eps}$ for all $t \geq 0$.
\end{proof}

\begin{figure}
  \includegraphics[width=\textwidth]{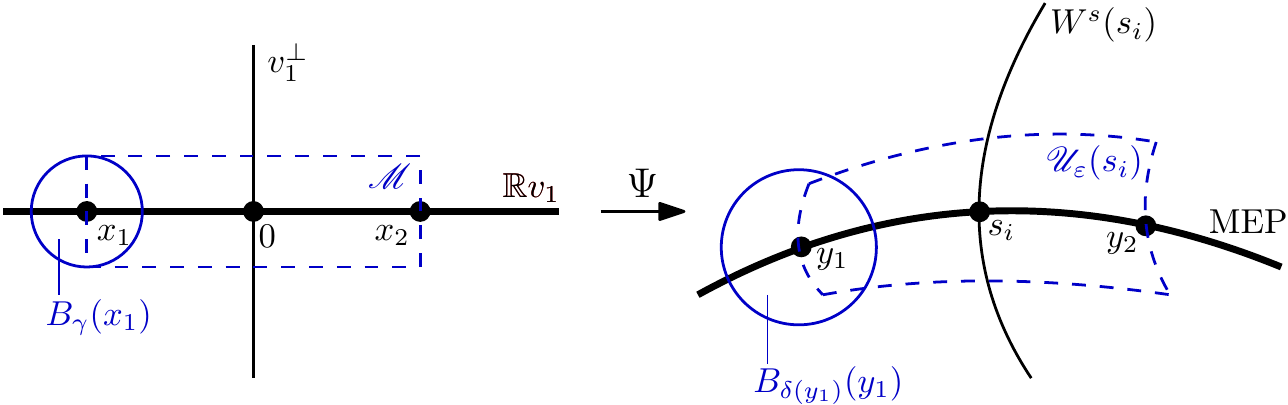}
  \caption{An illustration of the use of the Hartman--Grobman theorem in the proof of Lemma~\ref{lem: construction of stable neighborhoods}.}
  \label{fig: hartman grobman}
\end{figure}

Uniform stability of $\m$ under gradient descent in a one-sided distance is an immediate corollary of Lemma~\ref{lem: construction of stable neighborhoods}:

\begin{lemma}
  \label{cor: uniform stability}
  For every $\eps>0$, there exists a $\delta>0$ so that
  \begin{equation*}
    d(x, \m) \leq \delta \text{ implies } d(S_t(x), \m) \leq \eps \text{ for all } t \geq 0.
  \end{equation*}
\end{lemma}

\begin{proof}
  Let $\mathscr{U}_\eps(p)$ be defined as in Lemma~\ref{lem: construction of stable neighborhoods}. The open set
  \begin{equation*}
    \mathscr{U} := \cup_{p \in \m} \mathscr{U}_\eps(p)
  \end{equation*}
  contains $\m$, and $x \in \mathscr{U}$ implies $d(S_t(x), \m) \leq \eps$ for all $t \geq 0$. Now for any $x \in \m$, define
  \begin{equation*}
    r(x) := \sup \{\rho > 0 : B_\rho(x) \subset \mathscr{U} \}.
  \end{equation*}
  The function $r(x)$ is lower semicontinuous, and $\m$ is compact, so $r(x)$ attains a minimum $\delta$ on $\m$. We have $\delta >0$ since $\mathscr{U}$ is open. Moreover,  $d(x, \m) \leq \delta$ implies $x \in \mathscr{U}$, so $d(S_t(x), \m) \leq \eps$ for all $t \geq 0$, as desired.
\end{proof}

Lemma~\ref{lem: construction of stable neighborhoods} also implies asymptotic stability of $\m$ under gradient descent in a one-sided distance:

\begin{lemma}
  \label{cor: asymptotic stability of gradient descent}
  There exist an open set $B$ with $\m \subset B$ and a function ${T:(0,\infty) \rightarrow (0,\infty)}$ so that for any $\eps >0$ and $x \in B$,
  \begin{equation*}
    d(S_t(x), \m) \leq \eps \text{ whenever } t \geq T(\eps).
  \end{equation*}
\end{lemma}

\begin{proof}
  For each $p \in \m$, let $\U_{1}(p)$ be the set constructed in the proof of Lemma~\ref{lem: construction of stable neighborhoods} with $\eps=1$.
  Define
  \begin{equation*}
    \U:= \cup_{p \in \{m_1, \dots, m_k, s_1, \dots, s_{k-1}\}} \U_{1}(p).
  \end{equation*}
  We claim that for some time $T$ sufficiently large, the set
  \begin{equation}\label{eq: definition of set for asymptotic stability}
    B' := \cup_{t=0}^T S_{-t}(\U)
  \end{equation}
  contains $\m$.
  To see this, we observe that $\cup_{t=0}^\infty S_{-t}(\U)$ is an open cover of $\m$. The sets $S_{-t}(\U)$ are open, since $\U$ is open and the flow $S_{t}$ is a diffeomorphism. Moreover, these sets cover $\m$, since they contain the stationary points, and for each ${p \in \m}$, the trajectory $S_t(p)$ converges to a stationary point. Therefore, since $\m$ is compact it follows that $B'$ contains $\m$ for some finite $T>0$.

  Now fix $\eps >0$. Let $B$ be an bounded open subset of $B'$ with $\m \subset B$ and such that $\bar B$, the closure of $B$, is contained in $B'$.
  By Lemma~\ref{cor: uniform stability}, there exists an open neighborhood $\mathcal{V}_\eps$ of $\m$ so that $x \in \mathcal{V}_\eps$ implies $d(S_t(x),\m) < \eps$ for all $t\geq 0$.
  For any $x \in B$, define
  \begin{equation*}
    t(x;\eps) := \inf \{t \geq 0: S_t(x) \in \mathcal{V}_\eps\}.
  \end{equation*}
  Observe that $t(x;\eps)$ is finite, since when $x \in B$, $S_T (x) \in \U$, and so the trajectory $S_t(x)$ converges to one of the stationary points.
  We claim that $t(x;\eps)$ is upper semicontinuous (as a function of $x$) on $\bar B$.
  By definition of $t(x;\eps)$, for any $\eta >0$, there exists $\eta' \in [0,\eta)$ so that $S_{t(x;\eps)+\eta'}(x) \in \mathcal{V}_\eps$.
  Moreover, since $\mathcal{V}_\eps$ is open and $S_{t(x;\eps)+\eta'}$ is continuous, there exists some $\delta >0$ small enough that $S_{t(x;\eps)+\eta'}(B_\delta(x)) \subset V_\eps$.
  Thus, for any $y \in B_\delta(x)$,
  \begin{equation*}
    t(y;\eps) \leq t(x;\eps) + \eta' \leq t(x;\eps) + \eta,
  \end{equation*}
  and so $t(x;\eps)$ is upper semicontinuous.
  Therefore, $t(x;\eps)$ attains a maximum $T(\eps)$ on the compact set $\bar B$, hence
  \begin{equation*}
    d(S_t(x),\m) < \eps \text{ whenever } t \geq T(\eps),
  \end{equation*}
  which completes the proof.
\end{proof}

Finally, by Lemma~\ref{lem: bound on distance gives bound on haus distance}, uniform and asymptotic stability for gradient descent imply the analogous stability properties in Hausdorff distance for the dynamics on curves.

\begin{theorem}\label{thm: uniform and asymptotic stabilit of mep}
  $\m$ is uniformly stable in the Hausdorff distance under the gradient descent dynamics on curves. Moreover, $\m$ is asymptotically stable with uniform convergence on $B$ for some compact set $B$ containing an open neighborhood of $\m$.
\end{theorem}

\begin{proof}
  We prove only uniform stability; the proof of asymptotic stability is similar. Let $\eps >0$. By Lemma~\ref{lem: bound on distance gives bound on haus distance}, there exists $\eta >0$ so that $d(\phi, \m) < \eta$ implies $\hd(\phi, \m) < \eps$. Moreover, by Lemma~\ref{cor: uniform stability}, there exists $\delta >0$ so that $d(\phi, \m) < \delta$ implies $d(S_t(\phi), \m) < \eta$ for all $t \geq 0$. Therefore, if  $\phi \in \curves(m_1, m_k)$ and $\hd(\phi,\m) < \delta$, we have $\hd(S_t(\phi),\m) < \eps$ for all $t\geq 0$. Thus, $\m$ is uniformly stable in Hausdorff distance under GDDC.
\end{proof}

Our stability results guarantee the existence of a Lyapunov function.
We begin by constructing an appropriate domain, $B$,  for this Lyapunov function. The key property of this domain is forward invariance, i.e.\@ that $x \in B$ implies $S_t(x) \in B$ for all $t \geq 0$.

\begin{lemma}\label{lem: domain of lyapunov function}
  For some $\delta >0$ there exists a compact, forward invariant set $B$ containing an open neighborhood of $\m$.
\end{lemma}
\begin{proof}
  By Lemma~\ref{cor: uniform stability}, for any $\eps >0$ there exists a $\delta >0$ so that $d(x, \m) < \delta$ implies $d(S_t(x), \m) < \eps$ for all $t\geq 0$.
  Define
  \begin{equation*}
    \NN_\delta(\m) := \cup_{p \in \m} B_\delta (p),
  \end{equation*}
  and
  \begin{equation*}
    B := \overline{\cup_{\tau \geq 0} S_\tau(\NN_\delta(\m))}.
  \end{equation*}
  We have $B \supset \NN_\delta(\m)$.
  Moreover,
  \begin{equation*}
    B \subset \overline{\NN_\eps(\m)},
  \end{equation*}
  so $B$ is bounded.

  To see that $B$ is forward invariant, let $x \in B$.
  We must show that for any $t\geq 0$, $S_t x$ is a limit point of $\cup_{\tau \geq 0} S_\tau(\NN_\delta(\m))$.
  We first observe that by definition, for any $\eps >0$, there is some $x(\eps) \in \cup_{\tau \geq 0} S_\tau(\NN_\delta(\m))$ with
  \begin{equation*}
    \lvert x - x(\eps) \rvert < \eps.
  \end{equation*}
  Therefore, for any $\eps >0$,
  \begin{equation*}
    \lvert S_t x - S_t x(\eps \exp(-tL)) \rvert
    \leq \exp(tL) \lvert x - x(\eps \exp(-tL)) \rvert
    \leq \eps.
  \end{equation*}
  Thus, since $S_t x(\eps \exp(-tL)) \in \cup_{\tau \geq 0} S_\tau(\NN_\delta(\m))$, $S_t x$ is a limit point, as desired.
\end{proof}

We now construct a Lyapunov function for $\m$ in Hausdorff distance.

\begin{theorem}\label{thm: existence of Lyapunov function}
  Let $B$ be a compact, forward invariant set containing an open neighborhood of $\m$. Assume that $\m$ is asymptotically stable with uniform convergence on $B$. There exists a Lyapunov function $V: \curves(m_1,m_k,B) \rightarrow [0,\infty)$ for $\m$.
\end{theorem}

\begin{proof}
  We adapt the proof of~\cite[Theorem~1.7.6]{StuartHumphries:DynamicalSysNumerAn1996}. Details appear in Appendix~\ref{apx: proof of existence of lyapunov function}.
\end{proof}

\begin{remark}
  \label{rem: comparison with cameron article}
  Previous work has analyzed convergence of trajectories of GDDC to $\mep$s~\cite{cameron_string_2011}. It is known that if the limit set of a trajectory is a $\mep$, then the trajectory converges to that $\mep$~\cite[Theorem~3]{cameron_string_2011}. This is the case if there are finitely many critical points of the potential $V$ and all are minima or saddles of index one~\cite[Corollary~4]{cameron_string_2011}, or if both the potential and the initial curve are piecewise analytic~\cite[Corollary~7]{cameron_string_2011}. However, we are not aware of any results other than our Theorem~\ref{thm: uniform and asymptotic stabilit of mep} regarding the local asymptotic and uniform stability of individual $\mep$s. It is these local stability results which imply that a given MEP can be computed using a discretization of GDDC.
\end{remark}

\section{Convergence of the String Method}

Our main result in this section is Theorem~\ref{thm: convergence}, which implies that any $\mep$ $\m$ passing alternately through minima and saddles of index one may be approximated to arbitrary accuracy by the string method. The proof uses the existence of a Lyapunov function to show that discretization errors in the string method do not accumulate. Thus, the string method follows the GDDC, converging to a neighborhood of $\m$ over long times.

For convenience, we analyze only the simplified and improved string method with the linear interpolant.
\begin{assumption}
  Assume that $\I$ is the linear interpolant defined in~\eqref{eq: linear interpolant}.
\end{assumption}

We also place consistency and stability assumptions on the numerical integrator:

\begin{assumption}
  \label{asm: truncation error}
  For some $q \geq 2$ and $D >0$, we have
  \begin{equation}
    \label{eq: truncation error of integrator}
    \lvert S_{\Delta t} (x) - \bar S_{\Delta t}(x) \rvert \leq D \Delta t^q
  \end{equation}
  for all $x \in \Real^d$.
  In addition,
  \begin{equation}
    \label{eq: lipschitz property of integrator}
    \lvert \bar S_{\Delta t}^n (x) - \bar S_{\Delta t}^n (y) \rvert \leq \exp(Ln\Delta t) \lvert x - y\rvert.
  \end{equation}
\end{assumption}

We note that all commonly used integrators satisfy~\eqref{eq: truncation error of integrator} with $q \geq 2$, and many integrators satisfy~\eqref{eq: lipschitz property of integrator}; see~\cite{StuartHumphries:DynamicalSysNumerAn1996}. In particular, for Euler's method,
\begin{align*}
  \lvert \bar S_{\Delta t} (x) - \bar S_{\Delta t}( y) \rvert &= x-y - \Delta t(-\nabla V(x) - \nabla V(y)) \leq (1 + \Delta t L) \lvert x - y\rvert.
\end{align*}
Therefore, since $1 + \Delta t L \leq \exp(\Delta t L)$,
\begin{equation*}
  \lvert \bar S_{\Delta t}^n( x) - \bar S_{\Delta t}^n (y) \rvert \leq (1 + \Delta t L)^n \lvert x - y\rvert \leq  \exp(Ln\Delta t) \lvert x - y\rvert.
\end{equation*}

We now estimate the error introduced when reparametrizing the string.
\begin{lemma}
  \label{lem: interp error}
  Let $x \in \strings^{N+1}$.
  We have
  \begin{displaymath}
    \hd(\I x, \I \R x) \leq \frac{m(x)}{2}.
  \end{displaymath}
\end{lemma}
\begin{proof}
  First, we observe that for any $x \in \strings^{N+1}$, we have $m(Rx) \leq m(x)$. To see this, let $s(\I x, a, b)$ denote the arc-length along $\I x$ between $a$ and $b$. We have
  \begin{equation*}
    |Rx_i - Rx_{i+1}| \leq s(\I x, \R x_i, \R x_{i+1}) = \frac{\l(x)}{N}
    = \frac{\sum_{i=0}^{k-1} \lvert x_{i+1} - x_i \rvert}{N} \leq m(x).
  \end{equation*}
  Now let $y \in \I \R x$, and suppose that $y$ is between $\R x_i$ and $\R x_{i+1}$. We have
  \begin{displaymath}
    \min \{|\R x_i - y|, |\R x_{i+1} - y|\} \leq \frac{1}{2} |\R x_i - \R x_{i+1}|
    \leq \frac{m(Rx)}{2} \leq \frac{m(x)}{2}.
  \end{displaymath}
  Therefore, since $Rx_i \in Ix$ and $Rx_{i+1} \in Ix$, $\dist (\I \R x, \I x) \leq \frac{m(x)}{2}$.
  Now let  $y \in \I x$, and suppose that $y = I(\lstar(x), x)(\alpha)$ for some $\alpha \in \left [ \frac{i}{N}, \frac{i+1}{N} \right ]$. We have
  \begin{displaymath}
    \min \{|\R x_i - y|, |\R x_{i+1} - y|\} \leq \frac{1}{2} s(\I x, \R x_i, \R x_{i+1}) \leq \frac{m(x)}{2},
  \end{displaymath}
  so $\dist (\I x, \I \R x) \leq \frac{Kh}{2}$.
  We conclude that $\hd(\I x, \I \R x) \leq \frac{m(x)}{2}$.
\end{proof}

We now show that reparametrization occurs only after each image has been advanced at least by a certain time $\tmin>0$. This is crucial for the numerical analysis, since the upper bound on reparametrization error in Lemma~\ref{lem: interp error} does not have a factor involving the time step $\Delta t$. This upper bound suffices because we do not consider whether the string method approximates GDDC in the limit as $\Delta t \rightarrow 0$. Instead, we are interested only in the convergence of strings to $\m$ over long times.

\begin{lemma}\label{lem: bound on interp time}
  \ignore{  Recall the definitions of $h$, $K$, and $m$ from Section~\ref{sec: algorithm},
    and let $L$ be a Lipschitz constant for $\grad V$.}
  Let $x \in \strings^{N+1}$.
  Define
  \begin{equation*}
    \tmin := \frac{\log(K)}{L}
  \end{equation*}
  We have
  \begin{equation*}
    m(\bar S^n_{\Delta t}(x) )\leq K m(x) \text{ for all }  0 \leq n \leq \left \lfloor \frac{\tmin}{\Delta t} \right \rfloor .
  \end{equation*}
\end{lemma}
\begin{proof}
  By~\eqref{eq: lipschitz property of integrator}, we have
  \begin{equation*}
    \lvert \bar S^n_{\Delta t} (x_{i+1}) -\bar S^n_{\Delta t}(x_{i}) \rvert \leq \exp(L n \Delta t)  \lvert x_{i+1} - x_{i} \rvert \leq \exp(L \tmin)m(x) \leq  K m(x)
  \end{equation*}
  for all $i = 1, \dots, k$.
\end{proof}

Finally, we estimate the error resulting from evolving only the images $x_i$ by gradient descent in each step, not the entire interpolated curve $Ix$.

\begin{lemma}\label{lem: error in commuting evolution and interpolation}
  Let $x \in \strings^{N+1}$ and let $\tmin$ be defined as in Lemma~\ref{lem: bound on interp time}. There exists a constant $C>0$ depending only on $V$ and $\tmin$ so that
  \begin{equation*}
    \hd(S_{\Delta t} I x, I S_{\Delta t} x ) \leq C m(x)^2  \Delta t \text{ for all } 0 \leq \Delta t \leq \tau.
  \end{equation*}
\end{lemma}

\begin{proof}
  Let $x \in \strings^{N+1}$. Fix $\alpha \in [0,1]$ and $i \in \{0, \dots, N\}$. Let
  \begin{equation*}
    \bar x = (1-\alpha)x_i + \alpha x_{i+1}.
  \end{equation*}
  Write
  \begin{equation*}
    y =x_{i+1}-x_i.
  \end{equation*}
  For any $z \in \Real^d$, define $\partial_z$ to be the derivative in direction $z$, so for example
  \begin{equation*}
    \partial_z S_t(x) = \left . \frac{d}{d\eps} \right \rvert_{\eps=0} S_t(x+\eps z).
  \end{equation*}

  Since $\grad V$ is twice continuously differentiable, the flow $S_t(x)$ is also twice continuously differentiable as a function of $(t,x)$ by~\cite[Theorem~1.3]{Chicone:ODEsApplications2006}.
  Therefore, by Taylor's theorem, for some $x'$ and $x''$ in the convex hull of $\{x_i, x_{i+1}\}$, we have
  \begin{align*}
    &S_{\Delta t}( x_i) - S_{\Delta t} (\bar x)  =-\alpha \partial_y S_{\Delta t}(\bar x) + \frac12 \alpha^2 \partial_y^2 S_{\Delta t} (x'),
  \end{align*}
  and
  \begin{align*}
    &S_{\Delta t}( x_{i+1}) - S_{\Delta t} (\bar x) =(1-\alpha) \partial_y S_{\Delta t}(\bar x) + \frac12 (1-\alpha)^2 \partial_y^2 S_{\Delta t} (x'').
  \end{align*}
  It follows that
  \begin{align}
    \lvert (1-\alpha) S_{\Delta t} x_i + \alpha S_{\Delta t} x_{i+1} - S_{\Delta t} (\bar x ) \rvert
    &\leq \frac12 \sup_{z \in \Real^d} \left  \lvert \partial_{y/ \lvert y \rvert}^2 S_{\Delta t} (z) \right \rvert \lvert y \rvert^2 \nonumber \\
    &\leq \frac12 \sup_{z \in \Real^d} \left  \lvert \partial_{w}^2 S_{\Delta t} (z) \right \rvert m(x)^2, \label{eq: taylor theorem upper bound on interpolation error}
  \end{align}
  where we define $w = y/\lvert y \rvert$.

  $S_t(x)$ solves the gradient descent initial value problem
  \begin{equation*}
    \begin{dcases}
      \frac{d}{dt}S_t(x) &= -\grad V(S_t(x)) \\
      S_0(x) &=x.
    \end{dcases}
  \end{equation*}
  Differentiating this initial value problem twice with respect to $x$ yields
  \ignore{ \begin{equation}\label{eq: ivp for second derivative of flow}
      \begin{dcases}
        \frac{d}{dt} \partial_w S_t(x)_q &= - \sum_{\ell=1}^d \frac{\partial^2 V}{\partial x_q \partial x_\ell}(S_t(x)) \partial_w S_t(x)_\ell \\
        \partial_w S_0(x) &= w \\
        \frac{d}{dt} \partial_w^2 S_t(x)_q
        &= -\sum_{\ell,m=1}^d \frac{\partial^3 V}{\partial x_q \partial x_\ell \partial x_m}(S_t(x)) \partial_w S_t(x)_\ell \partial_w S_t(x)_m\\
        &\phantom{=}- \sum_{\ell=1}^d \frac{\partial^2 V}{\partial x_q \partial x_\ell}(S_t(x)) \partial_w^2 S_t(x)_\ell \\
        \partial_w^2 S_0(x) &= 0.
      \end{dcases}
    \end{equation}}
  \begin{equation}\label{eq: ivp for second derivative of flow}
    \begin{dcases}
      \frac{d}{dt} \partial_w^2 S_t(x) &= \Gamma(t) + A(t) \partial_w^2 S_t(x) \\
      \partial_w^2 S_0(x) &= 0.
    \end{dcases}
  \end{equation}
  where
  \begin{equation*}
    \begin{aligned}
      \Gamma(t)_q  &:= -\sum_{\ell,m=1}^d \frac{\partial^3 V}{\partial x_q \partial x_\ell \partial x_m}(S_t(x)) \partial_w S_t(x)_\ell \partial_w S_t(x)_m &&\text{ for } q=1, \dots, d \text{ and } \\
      A(t)_{q\ell} &:=-  \frac{\partial^2 V}{\partial x_q \partial x_\ell}(S_t(x)) &&\text{ for } q,\ell = 1, \dots, d. \\
    \end{aligned}
  \end{equation*}
  (Here, $\partial_w S_t(x)_\ell$ and $\partial_w^2 S_t(x)_\ell$ denote the $\ell$'th coordinate components of  $\partial_w S_t(x)_\ell$ and  $\partial_w^2 S_t(x)_\ell$, respectively.)

  Since $\partial_w^2 S_0(x)=0$, we have $\lvert \partial_w^2 S_t(x) \rvert \leq C(x)t$ for any fixed $x \in \Real^d$. However, it is not immediately clear that this estimate holds with a constant independent of $x$.
  To prove this, observe that since $\grad V$ is Lipschitz with constant $L$, we have
  \begin{equation*}
    \sup_{t \in [0, \infty) }\lVert A(t) \rVert_2 \leq  \sup_{z \in \Real^d}  \left \lVert \frac{\partial^2 V}{\partial x_q \partial x_\ell}(z) \right \rVert_2 \leq L,
  \end{equation*}
  where $\lVert \cdot \rVert_2$ is the operator norm induced by Euclidean distance $\lvert \cdot \rvert$.
  Moreover, $S_t$ is Lipschitz with constant $\exp(Lt)$, so
  \begin{equation*}
    \partial_w S_t(x) \leq \exp(Lt) \lvert w \rvert =  \exp(Lt).
  \end{equation*}
  By Assumption~\ref{asm: assumptions on V},
  \begin{equation*}
    \left \lvert \frac{\partial^3 V}{\partial x_q \partial x_\ell \partial x_m}(S_t(x)) \right \rvert \leq C_3.
  \end{equation*}
  Thus,
  \begin{equation*}
    \Gamma(t) \leq d C_3 \exp(2L t).
  \end{equation*}
  It then follows from~\eqref{eq: ivp for second derivative of flow} that
  \begin{align*}
    \lVert \partial_w^2 S_t(x) \rVert &\leq \int_0^t \lVert \Gamma(s) \rVert \, ds + \int_0^t \lVert A(s) \rVert_2 \lVert \partial_w^2 S_s(x) \rVert \, ds \\
    &\leq \frac{dC_3}{2L}(\exp(2Lt) -1) + L \int_0^t \lVert \partial_w^2 S_s(x) \rVert \, ds.
  \end{align*}
  Therefore, by Gr\"onwall's inequality,
  \begin{align}
    \lVert \partial_w^2 S_t(x) \rVert \leq \frac{dC_3}{2L}(\exp(2Lt) -1) \exp(Lt).\label{eq: gronwalls}
  \end{align}

  Finally, by~\eqref{eq: taylor theorem upper bound on interpolation error} and~\eqref{eq: gronwalls}, we have
  \begin{align*}
    \lvert (1-\alpha) S_{\Delta t} x_i + \alpha S_{\Delta t} x_{i+1} - S_{\Delta t} (\bar x ) \rvert &\leq \frac{dC_3}{2L}(\exp(2L\Delta t) -1) \exp(L\tmin) m(x)^2 \\
                                                                                                     &\leq dC_2 \exp(3L\tmin) m(x)^2 \Delta t \\
    &\leq dC_2 K^3 m(x)^2 \Delta t
  \end{align*}
  for all  $\Delta t \leq \tmin$.
  This proves the result, since every point $y \in I S_{\Delta t} x$ takes the form
  \begin{equation*}
    y = (1-\alpha) S_{\Delta t} x_i + \alpha S_{\Delta t} x_{i+1}
  \end{equation*}
  for some $i \in \{0, \dots, N\}$ and $\alpha \in [0,1]$, and every point $z \in S_{\Delta t} Ix$ takes the form
  \begin{equation*}
    z = S_{\Delta t}((1-\beta)x_j + \beta x_{j+1}).
  \end{equation*}
  for some $j \in \{0, \dots, N\}$ and $\beta \in [0,1]$.
\end{proof}

We now prove our convergence result for the simplified and improved string method. We note that some aspects of the proof were inspired by a result concerning  the persistence of attractors of ODEs under discretization; see~\cite{kloeden-lorenz-1986} and~\cite[Theorem~7.5.1]{StuartHumphries:DynamicalSysNumerAn1996}.

\begin{theorem} \label{thm: convergence}
  There exist $h_0 >0$, $r_0 >0$, $N_0 >0$, and a function $e:(0,h_0) \times (0, \tau) \rightarrow (0, \infty)$ with
  \begin{equation*}
    \lim_{h, \Delta t \rightarrow 0^+} e(h, \Delta t) = 0
  \end{equation*}
  such that if $\hd(I x^0, \m) < r_0$ and $h < h_0$, then
  \begin{equation*}
    \hd(I x^n, \m) \leq e(h, \Delta t)  \text{ for all } n > N_0.
  \end{equation*}
\end{theorem}
\begin{proof}
  First, one must show that there exist $h_0 >0$, $r_0 >0$, and $\tmax >0$ so that if  $h <h_0$, $\Delta t \leq \tmax$, and $\hd(Ix^0, \m) \leq r_0$, then $Ix^\ell$ belongs to the domain $\mathscr{W}$ of $W$ for all $\ell \geq 0$. We prove this in Appendix~\ref{apx:  existence of invariant nbhd for sm}.

  Once it is established that the entire trajectory of the string method lies in the domain of $W$, the result follows from a variation of parameters formula involving $W$.
  By the contraction property~\eqref{eq: contraction property of Lyapunov function} and Lipschitz property~\eqref{eq: Lipschitz property of Lyapunov function} of the Lyapunov function, we have
  \begin{align}
    W(I x^{\ell+1}) &\leq W(S_{\Delta t} I x^\ell) + \lvert  W(Ix^{\ell+1}) - W(S_{\Delta t} I x^\ell) \rvert \nonumber \\
                    &\leq \exp(-c \Delta t) W(I x^\ell) + \hd(Ix^{\ell+1}, S_{\Delta t} I x^\ell) \nonumber \\
                    &\leq \exp(-c \Delta t) W(I x^\ell) +  \hd(Ix^{\ell+1}, I \bar S_{\Delta t} x^\ell) \\
    &\qquad \qquad + \hd(I \bar S_{\Delta t} x^\ell, I S_{\Delta t} x^\ell) +  \hd(I S_{\Delta t} x^\ell, S_{\Delta t} I x^\ell) \nonumber \\
                    &\leq  \exp(-c \Delta t) W(I x^\ell) + R_\ell + T_\ell + C_\ell  \label{eq: single step decrease in lyapunov function},
  \end{align}
  where we define
  \begin{align*}
    R_\ell &:= \hd(Ix^{\ell+1}, I S_{\Delta t} x^\ell), \\
    T_\ell &:=  \hd(I \bar S_{\Delta t} x^\ell, I S_{\Delta t} x^\ell), \text{ and } \\
    C_\ell &:= \hd(I S_{\Delta t} x^\ell, S_{\Delta t} I x^\ell).
  \end{align*}
  The term $R_\ell$ is associated with reparametrization, $T_{\ell}$ is the truncation error of the numerical integrator, and $C_\ell$ is associated with commuting evolution by $S_{\Delta t}$ and interpolation.

  By induction using~\eqref{eq: single step decrease in lyapunov function}, we have the variation of parameters formula
  \begin{align}\label{eq: variation of parameters formula}
    W(Ix^\ell) &\leq \exp(-c \ell \Delta t) W(Ix^0) + \sum_{j=0}^{\ell-1} \exp(-c(\ell-j-1) \Delta t) (R_j +T_j + C_j).
  \end{align}
  Assumption~\ref{asm: truncation error} implies
  \begin{equation*}
    T_\ell \leq D \Delta t^q,
  \end{equation*}
  since when $I$ is the linear interpolant
  \begin{equation*}
    \hd(Ix , Iy) \leq \max_{i=0, \dots, N+1} \lvert x_i - y_i \rvert
  \end{equation*}
  for any $x, y \in \strings^{N+1}$.
  Moreover, by Lemma~\ref{lem: error in commuting evolution and interpolation},
  \begin{equation*}
    C_\ell \leq C(Kh)^2 \Delta t.
  \end{equation*}
  Therefore,
  \begin{align}
    \sum_{j=0}^{\ell-1} \exp(-c(\ell-j-1) \Delta t) (T_j + C_j) &\leq  \frac{(D \Delta t^q+ C(Kh)^2) \Delta t (1-\exp(-c\ell \Delta t))}{1-\exp(-c\Delta t)}\nonumber \\
                                                           &\leq  \frac{(D \Delta t^{q-1} + C(Kh)^2) (1-\exp(-c\ell \Delta t))}{c} \nonumber \\
                                                           &\leq  \frac{D \Delta t^{q-1}+C(Kh)^2}{c}. \label{eq: estimate of error due to commutation}
  \end{align}
  (The second inequality above follows since $x \mapsto 1- \exp(-cx)$ is a concave fucntion with the tangent line $x \mapsto cx$ at $x=0$, so $1- \exp(-c\Delta t) \leq c\Delta t$.)

  The term $R_\ell$ is zero unless reparametrization occurs in the $\ell+1$'st step of the string method.
  Let $r_i$ be the step in which the $i$'th reparametrization occurs.
  By Lemma~\ref{lem: bound on interp time},
  \begin{equation*}
    r_i \geq \left \lceil \frac{\tmin}{\Delta t} \right \rceil i.
  \end{equation*}
  When reparametrization does occur in the $\ell+1$'st step, we have $m(x^\ell) \leq Kh$, but $m(x^{\ell+1}) > Kh$. However, since $S_{\Delta t}$ is Lipschitz with constant $\exp(L\Delta t)$,
  \begin{equation*}
    m(x^{\ell+1}) \leq \exp(L \Delta t) K h.
  \end{equation*}
  Therefore, by Lemma~\ref{lem: interp error}
  \begin{equation*}
    R_\ell \leq \frac{\exp(L \Delta t) K h}{2}.
  \end{equation*}
  Thus, we have
  \begin{align}
    \sum_{j=0}^{\ell-1} \exp(-c(\ell-j-1) \Delta t) R_j &\leq \sum_{j=0}^{\infty} \exp(-c(\ell-j-1) \Delta t) R_j \nonumber \\
                                                           &\leq \frac{\exp(L \Delta t) K h}{2} \sum_{i=1}^\infty \exp \left (-c i \Delta t r_i \right ) \nonumber \\
                                                           &\leq \frac{\exp(L \Delta t) K h}{2} \sum_{i=1}^\infty \exp \left (-c i \left \lceil \frac{\tmin}{\Delta t} \right \rceil \Delta t \right ) \nonumber \\
                                                           &\leq \frac{\exp(L \Delta t) K h}{2} \frac{1}{1-\exp(-c\tmin)}. \label{eq: estimate of the error due to reparametrization}
  \end{align}

  We conclude that
  \begin{align*}
    W(Ix^\ell) &\leq \exp(-c \ell \Delta t) W(Ix^0) + \frac{\exp(L \Delta t) K h}{2(1-\exp(-c\tmin))} + \frac{D \Delta t^{q-1}+C(Kh)^2}{c} \\
               &\leq  \exp(-c \ell \Delta t) r_0 + \frac{\exp(L \Delta t) K h}{2(1-\exp(-c\tmin))} + \frac{D \Delta t^{q-1}+C(Kh)^2}{c}.
  \end{align*}
  (The second inequality above follows from the Lipschitz property of $W$~\eqref{eq: Lipschitz property of Lyapunov function} and ${W(\m)=0}$ using that $\hd(Ix^0, \m) \leq r_0$.)
  Therefore, for some $N_0 >0$ depending only on $r_0$, $\ell \geq N_0$ implies
  \begin{equation*}
    W(Ix^\ell) \leq  \frac{\exp(L \Delta t) K h}{(1-\exp(-c\tmin))} + \frac{2(D \Delta t^{q-1}+C(Kh)^2)}{c}.
  \end{equation*}
  Thus, for $\ell \geq N_0$,
  \begin{equation*}
    \hd(Ix^\ell, \m) \leq \alpha^{-1} \left (\frac{\exp(L \Delta t) K h}{(1-\exp(-c\tmin))} + \frac{2(D \Delta t^{q-1}+C(Kh)^2)}{c} \right ) =: e(h,\Delta t)
  \end{equation*}
  by~\eqref{eq: equivalence with distance of Lyapunov function} and the monotonicity of $\alpha$. We observe that since $\alpha(0) = 0$ and $\alpha$ is strictly increasing
  \begin{equation*}
    \lim_{h, \Delta t \rightarrow 0^+} e(h,\Delta t) = 0,
  \end{equation*}
  which concludes the proof.
\end{proof}

\appendix

\section{Proof of Lemma~\ref{lem: bound on distance gives bound on haus distance}}
\label{apx: proof of one-sided distance lemma}

\begin{figure}
  \includegraphics[width=\textwidth]{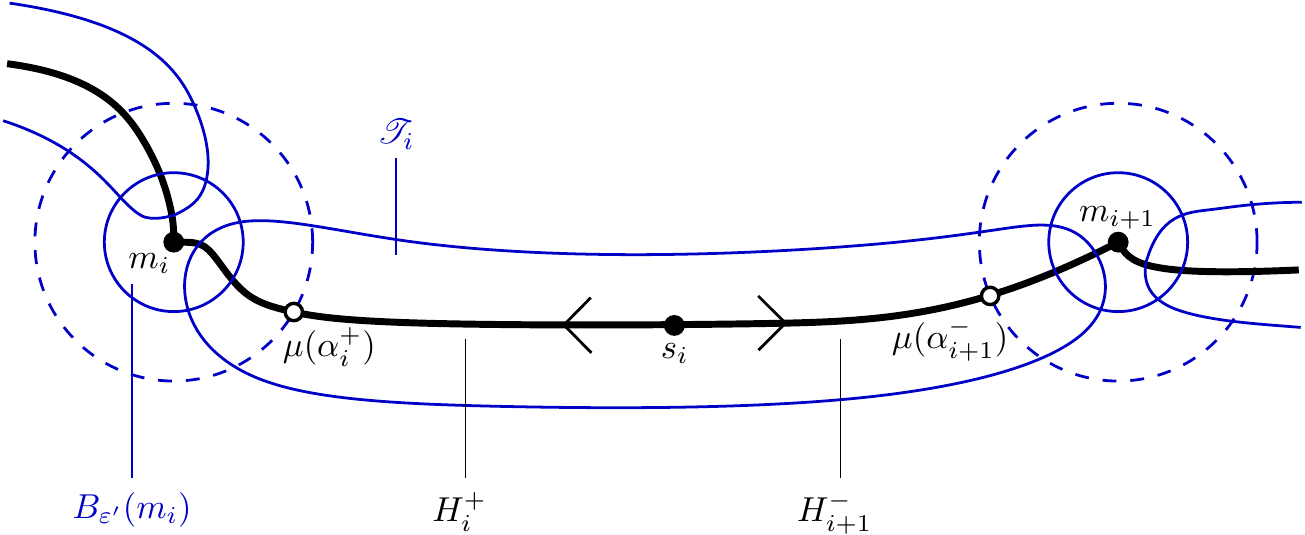}
  \caption{A depiction of the neighborhoods constructed in the proof of Lemma~\ref{lem: bound on distance gives bound on haus distance}. The bold black line represents $\m$. The arrows point in the direction of gradient descent along the heteroclinics $H_i^+$ and $H_i^-$ connecting $s_i$ to $m_i$ and $m_{i+1}$. The blue dashed balls are $B_{2 \eps'}(m_i)$ and $B_{2 \eps'}(m_{i+1})$.}
  \label{fig: distance lemma proof}
\end{figure}
\begin{proof}
  We adapt the argument based on tubular neighborhoods outlined after the statement of Lemma~\ref{lem: bound on distance gives bound on haus distance}. To allow for kinks at the minima, we devise an analogue of a tubular neighborhood consisting of tubes where $\m$ is smooth connected by balls surrounding the minima as in Figure~\ref{fig: distance lemma proof}.

  Let $\eps >0$.
  Since each minimum $m_i$ is linearly stable, there exists $\eps_i >0$ so that $x \in B_{\eps_i}(m_i)$ implies $\lvert S_t x - m_i \rvert$ is strictly decreasing with $t$.
  Let
  \begin{equation*}
    \eps' = \min \left \{\frac{\eps}{6}, \frac{\min_{i=1, \dots, k} \eps_i}{2} \right \}.
  \end{equation*}
  Observe that the balls $B_{2\eps'}(m_i)$ are disjoint, since the basins of attraction of the minima $m_i$ under the gradient flow are disjoint.

  Let $H_i^+$ be the heteroclinic connecting $m_i$ with $s_i$, and $H_{i+1}^-$ the heteroclinic connecting $s_i$ with $m_{i+1}$, as in Figure~\ref{fig: mep notation}. We claim that for any $i=1,\dots, k-1$, the curve segment
  \begin{equation*}
    M_i = H_i^+ \cup H_{i+1}^- \setminus B_{\frac{\eps'}{2}}(m_i) \cup B_{\frac{\eps'}{2}}(m_{i+1})
  \end{equation*}
  is a $\C^2$-embedded submanifold of $\Real^d$.
  To prove this, observe that in an open neighborhood of the the saddle point $s_i$, $M_i$ coincides with the local unstable manifold of $s_i$.
  Since $V \in \C^3$, by the invariant manifold theorem\cite[Theorem~5.2]{shub_stability_1987}, the local unstable manifold is a $\C^2$-embedded submanifold of $\Real^d$.
  Now since $M_i$ lies on the pair of heteroclinics $H_i^+$ and $H_{i+1}^-$, for some time $T>0$, the whole of $M_i$ is contained in the image under $S_T$ of the part coinciding with the local unstable manifold.
  The flow map $S_T$ is a diffeomorphism,
  and since $V \in \C^3$ we have $S_T \in \C^2$ by~\cite[Theorem~1.3]{Chicone:ODEsApplications2006}.
  Thus, $M_i$ is a $\C^2$-embedded submanifold.

  It follows by~\cite[Theorem~10.19]{lee_manifolds_2002} that $M_i$ has a tubular neighborhood.
  (The statement of \cite[Theorem~10.19]{lee_manifolds_2002} assumes a $\C^\infty$-embedded submanifold. However, the proof requires only a $\C^2$-embedded submanifold.)
  In particular, for any radius $\rho_i >0$ small enough, there exist an open neighborhood $\mathscr{T}_i$ with ${M_i \subset \mathscr{T}_i}$ and a continuous retraction $r_i:\mathscr{T}_i \rightarrow M_i$ such that
  \begin{equation*}
    \lvert x - r_i(x) \rvert < \rho_i.
  \end{equation*}
  For all $i=1,\dots,k-1$, let $0< \rho_i <\eps'$ be small enough that a tubular neighborhood exists.
  Since the closures of the curve segments $M_i$ are disjoint, we may assume that the tubular neighborhoods $\mathscr{T}_i$ are disjoint, reducing $\rho_i$ if necessary. Moreover, we may assume that $\mathscr{T}_i \cap B_{2\eps'}(m_k) \neq \emptyset$ only for $k=i,i+1$.

  Since $\m$ does not intersect itself, it is homeomorphic to an interval. That is, there exists a continuous bijection $\mu:[0,1] \rightarrow \m$ with continuous inverse $\mu^{-1} : \m \rightarrow [0,1]$. (Here, we equip $\m$ with the subspace topology inherited from $\Real^d$.)
  We may assume that for some $0=\alpha_1 < \alpha_2 < \dots < \alpha_k=1$,
  \begin{equation*}
    \mu(\alpha_i)=m_i. 
  \end{equation*}
  For each $i=1, \dots, k-1$, define
  \begin{equation*}
    \alpha_i^+ := \sup  \{\alpha \in [0,1] : \mu(\alpha) \in B_{2\eps'}(m_i)\},
  \end{equation*}
  and let $\alpha_k^+=1$.
  Similarly, for $i=2, \dots, k$, define
  \begin{equation*}
    \alpha_{i+1}^- := \inf  \{\alpha \in [0,1] : \mu(\alpha) \in B_{2\eps'}(m_{i+1})\},
  \end{equation*}
  and let $\alpha_1^-=0$.
  We claim that
  \begin{equation}
    \mu(\alpha) \in B_{2\eps'}(m_i) \text{ iff } \alpha \in [\alpha_i^-, \alpha_i^+].
    \label{eq: main property of alphaplus}
  \end{equation}
  To prove~\eqref{eq: main property of alphaplus}, observe that by continuity of $\mu$, $\mu(\alpha_i^+) \in \partial B_{2\eps'}(m_i)$. In fact, since $\lvert S_t( x) - m_i \rvert$ is strictly decreasing on $B_{2\eps'}(m_i)$, the heteroclinic trajectory $H_i^+$ intersects $\partial B_{2\eps'}(m_i)$ only at $\mu(\alpha^+_i)$.
  If we had $\lvert \mu(\alpha) -m_i \rvert > 2 \eps'$ for some $ \alpha \in [\alpha_i, \alpha_i^+]$, the intermediate value theorem would imply the existence of a second distinct point of intersection between $H_i^+$ and $\partial B_{2\eps'}(m_i)$, since $\mu$ is a bijection. Therefore, $\mu([\alpha_i, \alpha_i^+]) \subset B_{2\eps'}(m_i)$. A similar argument involving the heteroclinic $H_i^-$ shows $\mu([\alpha_i^-, \alpha_i]) \subset B_{2\eps'}(m_i)$, and~\eqref{eq: main property of alphaplus} follows.

  We will now construct an analogue $\Pi$ of the retraction $r_i$ for the entire $\mep$ $\m$.
  For convenience, instead of mapping onto $\m$, $\Pi$ will take values in $[0,1]$, the domain of the parametrization $\mu$ of $\m$.
  We handle kinks at minima by letting $\Pi$ take a constant value on a ball surrounding each minimum.
  We begin by defining for each $i=1, \dots, k-1$ the three continuous functions
  \begin{alignat*}{3}
    & p_i^+: B_{\eps'}(m_i) \rightarrow \Real && \qquad \text{ by } p_i^+ (x) = \alpha_i^+, \\
    & p_i: \mathscr{T}_i \rightarrow \Real && \qquad \text{ by }  p_i(x) = \min \{\alpha_{i+1}^-,\max \{ \alpha_i^+, \mu^{-1} \circ r_i(x)\}\} , \text{ and } \\
    & p_{i+1}^-:B_{\eps'}(m_{i+1}) \rightarrow \Real && \qquad \text{ by } p_{i+1}^-(x) = \alpha_{i+1}^-. \\
  \end{alignat*}
  These functions will be the constituent parts of our retraction $\Pi$.
  We claim that $p_i^-$, $p_i$, and $p_i^+$ agree where their domains intersect.
  To verify this, it will suffice to show that
  \begin{equation*}
    \begin{aligned}
      \mu^{-1} \circ r_i(x) &\leq \alpha_i^+ &&\text{ for } x \in B_{\eps'}(m_i) \cap \mathscr{T}_i, \text{ and } \\
      \mu^{-1} \circ r_i(x) &\geq \alpha_{i+1}^- &&\text{ for } x \in B_{\eps'}(m_{i+1}) \cap \mathscr{T}_i.
    \end{aligned}
  \end{equation*}
  Let $x \in B_{\eps'}(m_i) \cap \mathscr{T}_i$.
  Then
  \begin{equation*}
    \lVert r_i(x) -m_i\rVert \leq \lVert r_i(x) - x \rVert + \lVert x - m_i \rVert
    \leq \rho_i+\eps' \leq 2\eps',
  \end{equation*}
  and so $r_i(x) \in B_{2\eps'}(m_i)$.
  Therefore, $\mu^{-1} \circ r_i(x) \leq \alpha^+_i$ by~\eqref{eq: main property of alphaplus}.
  Similarly, $x \in B_{\eps'}(m_{i+1}) \cap \mathscr{T}_i$ implies $\mu^{-1} \circ r_i(x) \geq \alpha_{i+1}^-$.
  Thus, the three functions agree on the intersections of their domains. Therefore, since the domains are open, they define a single continuous function $\pi_i$ on
  \begin{equation*}
    \mathscr{U}_i := B_{\eps'}(m_i) \cup \mathscr{T}_i \cup  B_{\eps'}(m_{i+1}).
  \end{equation*}

  We now patch together the functions $\pi_i$ to construct our retraction $\Pi$.
  First, define $\tilde \pi_i: \mathscr{U}_i \rightarrow [0,1]$ by
  \begin{equation*}
    \tilde \pi_i(x) = \pi_i(x) - \sum_{j=2}^i \alpha_j^+ - \alpha_j^-.
  \end{equation*}
  \ignore{
    Observe that $\tilde \pi_i$ takes values in
    \begin{equation*}
      \left [ \alpha_i^- - \sum_{j=2}^{i-1} \alpha_j^+ - \alpha_j^-, \alpha_{i+1}^+ -  \sum_{j=2}^{i} \alpha_j^+ - \alpha_j^- \right ]
    \end{equation*}
  }
  The functions $\tilde \pi_i$ are continuous, have open domains, and agree on the intersections of their domains.
  In fact, for $i=1,\dots, k-1$, the functions $\tilde \pi_i$ and $\tilde \pi_{i+1}$ both take the value
  \begin{equation*}
    \alpha_{i+1}^- - \sum_{j=2}^{i+1} \alpha_j^+ - \alpha_j^-.
  \end{equation*}
  on $B_{\eps'}(m_{i+1}) =  \mathscr{U}_i \cap \mathscr{U}_{i+1} $.
  The domains $\mathscr{U}_i$ and $\mathscr{U}_j$ do not intersect if $\lvert i-j \rvert \geq 2$.
  Therefore, the functions $\tilde \pi_i$ define a single continuous function $\Pi$ on
  \begin{equation*}
    \mathscr{U} := \cup_{i=1}^{K-1} \mathscr{U}_i.
  \end{equation*}

  Since $\mathscr{U}$ is open and $\m$ is compact, there exists some $\eta >0$ so that $d(x, \m) \leq \eta$ implies $x  \in \mathscr{U}$.
  We claim that the conclusion of the lemma holds for this $\eta$.
  To see this, let  $\phi \in \curves(m_1,m_k)$ with $d(\phi, \m) \leq \eta$.
  We have
  \begin{equation*}
    \Pi(m_1) = \alpha_1^+ \text{ and }
    \Pi(m_k) = \alpha_k^--\sum_{j=1}^{k-1} \alpha_j^+ - \alpha_j^-.
  \end{equation*}
  Therefore, by the intermediate value theorem,
  \begin{equation}\label{eq: values taken by Pi}
    \Pi(\phi) \supset  \left [\alpha_1^+,\alpha_k^--\sum_{j=1}^{k-1} \alpha_j^+ - \alpha_j^- \right ].
  \end{equation}

  Now let $x \in \m$. By~\eqref{eq: values taken by Pi}, there exists some $y \in \phi$ so that $\Pi(y)=\Pi(x)$.
  To complete the proof, we will show that $\lvert x - y \rvert \leq \eps$.
  We distinguish two cases: Either $\Pi(x) \neq \Pi(m_i)$ for any $i \in \{1, \dots, k\}$ or else $\Pi(x)= \Pi(m_i)$ for some $i \in \{1, \dots, k\}$. In the first case, $y$ lies in a tube $\mathscr{T}_i$ for some $i$, and $x = r_i(y)$. Therefore,
  \begin{equation*}
    \lvert x-y \rvert < \rho_i < \eps' < \eps.
  \end{equation*}
  In the second case, we have $\Pi(y)=\Pi(x)=\Pi(m_i)$. We claim that for any $z \in \mathscr{U}$,
  \begin{equation}
    \Pi(z) = \Pi(m_i) \text{ implies } z \in B_{3 \eps'}(m_i).
    \label{eq: level sets of Pi at minima}
  \end{equation}
  If so, then
  \begin{equation*}
    \lvert x - y \rvert \leq \lvert x - m_i \rvert + \lvert y-m_i \rvert \leq 3 \eps' + 3 \eps' \leq \eps.
  \end{equation*}

  To prove~\eqref{eq: level sets of Pi at minima}, we observe that if $\Pi(z) = \Pi(m_i)$, then $z \in \mathscr{U}_i \cup \mathscr{U}_{i-1}$. Suppose that $z \in \mathscr{U}_i$. On the one hand, if $z \in B_{\eps'}(m_i)$, then~\eqref{eq: level sets of Pi at minima} holds trivially. On the other hand, if $z \notin B_{\eps'}(m_i)$, then $z \in \mathscr{T}_i$. In that case, since $\Pi(z) = \Pi(m_i)$,
  \begin{equation*}
    \mu^{-1} \circ r_i(z) \in [\alpha_i, \alpha_i^+],
  \end{equation*}
  so by~\eqref{eq: main property of alphaplus}, $r_i(z) \in B_{2\eps'}(m_i)$. Therefore,
  \begin{equation*}
    \lvert z - m_i \rvert \leq \lvert z - r_i(z) \rvert + \lvert r_i(z) - m_i \rvert \leq \eps' + 2 \eps' = 3 \eps'.
  \end{equation*}
  This concludes the proof of~\eqref{eq: level sets of Pi at minima}, hence the proof of Lemma~\ref{lem: bound on distance gives bound on haus distance}.
\end{proof}

\section{Proof of Theorem~\ref{thm: existence of Lyapunov function}}
\label{apx: proof of existence of lyapunov function}

\begin{proof}
  Let $B \subset \Real^d$ be a compact, forward invariant set containing $\m$.
  We will construct a Lyapunov function for $\m$, defined on the set $\curves(m_1,m_2,B)$.

  Define $G_k : [0,\infty) \rightarrow [0,\infty)$ by
  \begin{equation*}
    G_k(x) := \max \{0,x-1/k\}.
  \end{equation*}
  Let $c >0$.
  For any $\phi \in \curves(m_1,m_2,B)$, define
  \begin{equation*}
    V_k(\phi) := \sup_{\tau \geq 0} \exp(c\tau) G_k(\hd(S_\tau(\phi), \m)).
  \end{equation*}
  We note that $V_k$ is finite, since asymptotic stability on $B$ implies $G_k(S_\tau(\phi)) = 0$ for sufficiently large $\tau$.
  We will use the functions $V_k$ for $k=1,\dots, \infty$ to construct a Lyapunov function.

  First, we observe that $V_k$ has the contraction property~\eqref{eq: contraction property of Lyapunov function} of a Lyapunov function, since
  \begin{align}
    V_k(S_t (\phi))
    &= \sup_{\tau \geq 0} \exp(c\tau) G_k(\hd(S_{t+\tau} (\phi), \m))
      \nonumber \\
    &= \exp(-ct) \sup_{\tau \geq 0} \exp(c(t+\tau)) G_k(\hd(S_{t+\tau} (\phi), \m)) \nonumber \\
    &= \exp(-ct) \sup_{\tau \geq t} \exp(c\tau) G_k(\hd(S_\tau (\phi), \m))
      \nonumber \\
    &\leq \exp(-ct) V_k(\phi).
      \label{eq: contraction property of Vk}
  \end{align}

  We now find a Lipschitz constant for $V_k$.
  $\m$ is asymptotically stable with uniform convergence on $B$, and so for each $k \in \mathbb{N}$ there exists $T(1/k)$ so that for any $\phi \in \curves(m_1,m_2,B)$,
  \begin{equation*}
    \hd (S_{t} \phi, \m) \leq 1/k \text{ whenever } T(1/k) \leq t.
  \end{equation*}
  Therefore, for any $\phi, \psi \in \curves(m_1,m_2,B)$, we have
  \begin{align*}
    \lvert V_k(\phi) - V_k(\psi) \rvert
    &\leq
      \sup_{\tau=0} \exp(c\tau)
      \lvert G_k(\hd(S_\tau \phi,\m)) - G_k(\hd (S_\tau \psi, \m)) \rvert \\
    &=
      \sup_{0 \leq \tau \leq T(1/k)} \exp(c\tau)
      \lvert G_k(\hd(S_\tau \phi,\m)) - G_k(\hd (S_\tau \psi, \m)) \rvert \\
    &\leq
      \sup_{0 \leq \tau \leq T(1/k)} \exp(c\tau)
      \lvert \hd(S_\tau \phi,\m) - \hd (S_\tau \psi, \m) \rvert \\
    &\leq \sup_{0 \leq \tau \leq T(1/k)} \exp(c\tau) \hd(S_\tau \phi, S_\tau \psi) \\
    &\leq \sup_{0 \leq \tau \leq T(1/k)}
      \exp(c\tau) \exp(L\tau) \hd(\phi, \psi) \\
    &\leq \exp((L+c) T(1/k)) \hd(\phi, \psi), \\
  \end{align*}
  It follows that
  \begin{equation}\label{eq: Lipschitz constant for Vk}
    L_k := \exp((L+c) T(1/k))
  \end{equation}
  is a Lipschitz constant for $V_k$.

  We claim that $V: \curves(m_1,m_2, B) \rightarrow [0,\infty)$ defined by
  \begin{equation*}
    V(\phi) := \sum_{k=1}^\infty \frac{L_k^{-1} V_k(\phi)}{2^k}
  \end{equation*}
  is a Lyapunov function for $\m$ on $B$.
  To see that the sum defining $V$ converges, we observe that
  \begin{equation*}
    0 \leq V_k(\phi) \leq V_k(\m) + L_k \hd (\phi,\m) \leq L_k \hd (\phi,\m),
  \end{equation*}
  which implies
  \begin{equation}\label{eq: upper bound on V by distance}
    V(\phi) \leq \hd (\phi,\m).
  \end{equation}
  The contraction property~\eqref{eq: contraction property of Lyapunov function} and Lipschitz bound~\eqref{eq: Lipschitz property of Lyapunov function} in the definition of Lyapunov function follow from~\eqref{eq: contraction property of Vk} and~\eqref{eq: Lipschitz constant for Vk}, respectively.
  The upper bound in property~\eqref{eq: equivalence with distance of Lyapunov function} follows from~\eqref{eq: upper bound on V by distance}.

  Thus, it remains only to show the lower bound in~\eqref{eq: equivalence with distance of Lyapunov function}; that is, we must show
  \begin{equation*}
    \alpha(\hd(\phi,\m)) \leq V(\phi)
  \end{equation*}
  for some continuous strictly increasing function $\alpha$ with $\alpha (0) = 0$.
  To see this observe that if $\hd(\phi,\m) \geq 1$, then
  \begin{align*}
    V(\phi) &\geq \frac{1}{2^2 L_2} G_2(\hd(\phi, \m)) \\
            &\geq \frac{1}{2^2 L_2} \left (\hd(\phi,\m) - \frac{1}{2} \right ) \\
            &\geq \frac{1}{2^2 \exp((c+L)T(1/2))} \frac{1}{2} \hd(\phi,\m).
  \end{align*}
  Similarly, if $1/(k-1) > \hd(\phi,\m) \geq 1/k$, then
  \begin{equation*}
    V(\phi) \geq \frac{1}{2^{k+1} \exp\left((c+L)T\left(\frac{1}{k+1}\right)\right)} \frac{1}{k+1} \hd(\phi,\m).
  \end{equation*}
\end{proof}

\section{Statement and Proof of Lemma~\ref{lem: existence of invariant nbhd for sm}}
\label{apx:  existence of invariant nbhd for sm}
\begin{lemma}\label{lem: existence of invariant nbhd for sm}
  There exist $r_0 >0$,  $h_0>0$, and $\tmax >0$ such that if $\hd(Ix^0, \m) \leq r_0$,  $m(x^0) \leq h_0$, and $\Delta t \leq t_0$, then $Ix^{\ell} \in \mathscr{W}$ for all $\ell \geq 0$.
\end{lemma}
\begin{proof}

  For $r > 0$, let
  \begin{equation*}
    \NN(r) := \{ \phi \in \curves(m_1,m_2) : \hd(\phi, \m) \leq r\}.
  \end{equation*}
  By construction, the domain $\mathscr{W}$ of $W$ contains $\NN(r_1)$
  for some $r_1>0$.
  Let
  \begin{equation*}
    r_0 := \alpha \left ( \frac{r_1}{2} \right ),  \text{ and define }
    \mathscr{M}:= \{\phi \in \mathscr{W} : W(\phi) \leq r_0\}.
  \end{equation*}
  We claim
  \begin{equation*}
    \NN(r_0) \subset \mathscr{M}
    \subset \NN(r_1) \subset \mathscr{W}.
  \end{equation*}
  To prove this, we observe that by~\eqref{eq: Lipschitz property of Lyapunov function} and~\eqref{eq: equivalence with distance of Lyapunov function},
  \begin{equation*}
    \alpha(x) \leq x
  \end{equation*}
  for all $x \in [0,\infty)$. Therefore, $r_0 \leq r_1/2$, so $\NN(r_0) \subset \NN(r_1) \subset \mathscr{W}$.
  Having established that $\NN(r_0)$ lies in the domain $\mathscr{W}$ of $W$, inequality~\eqref{eq: Lipschitz property of Lyapunov function} implies that $\NN(r_0) \subset \mathscr{M}$.
  Finally, if $\phi \in \mathscr{M}$, then $\alpha(\hd(\phi, \m)) \leq r_0$ by~\eqref{eq: equivalence with distance of Lyapunov function}. Thus,
  \begin{equation*}
    \hd(\phi, \m) \leq \alpha^{-1} (r_0) \leq \frac{r_1}{2} < r_1
  \end{equation*}
  by monotonicity of $\alpha$, hence $\mathscr{M} \subset \NN(r_1)$.

  We will show that for some $h_0 >0$ and $\tmax >0$, if $Ix^0 \in \mathscr{M}$,  $m(x^0) \leq h \leq h_0$, and $\Delta t \leq \tmax$, then $Ix^\ell \in \mathscr{M}$ for all $\ell \geq 0$.
  First, we prove that if $Ix^\ell \in \mathscr{M}$ with $m(x^\ell) \leq Kh$ and $m(S_{\Delta t} x^\ell) \leq Kh$, then $Ix^{\ell+1} \in \mathscr{M}$. That is, if $Ix^\ell \in \mathscr{M}$ and reparametrization does not occur in the $(\ell+1)$'st step of the string method, then $I x^{\ell+1} \in \mathscr{M}$. To see this, observe that
  \begin{align*}
    \alpha(\hd(S_{\Delta t} I x^\ell,\m)) \leq W(S_{\Delta t} I x^\ell)
    \leq  e^{-c \Delta t} W(I x^\ell) \leq e^{-c \Delta t} r_0,
  \end{align*}
  and so
  \begin{equation*}
    \hd(S_{\Delta t} I x^\ell, \m)
    \leq \alpha^{-1} (e^{-c \Delta t} r_0)
    \leq \alpha^{-1}(r_0) \leq \frac{r_1}{2}.
  \end{equation*}
  Therefore, for $h_0$ and $\tmax$ small enough that
  \begin{equation}\label{eq: first condition on h0}
    D \tmax^q + C (Kh_0)^2 \Delta t \leq \frac{r_1}{2},
  \end{equation}
  we have
  \begin{align*}
    \hd(I x^{\ell+1},\m) &= \hd(I \bar S_{\Delta t} x^\ell, \m) \\
                         &\leq \hd(I \bar S_{\Delta t} x^\ell, IS_{\Delta t} x^\ell ) + \hd (I S_{\Delta t} x^\ell, S_{\Delta t} I x^\ell) + \hd(S_{\Delta t} I x^n, \m) \\
                         &\leq  D \Delta t^q+ C (Kh)^2 \Delta t + \frac{r_1}{2} \\
                         &\leq r_1
  \end{align*}
  by Assumption~\eqref{eq: truncation error of integrator} and Lemma~\ref{lem: error in commuting evolution and interpolation}.
  Thus, $Ix^{\ell+1} \in \NN(r_1) \subset \mathscr{W}$.

  Having shown that $Ix^{\ell+1}$ lies in the domain of $W$, the argument leading to inequality~\eqref{eq: single step decrease in lyapunov function} yields
  \begin{align}\label{eq: lyapunov inequality, no reparametrization}
    W(I x^{\ell+1}) 
                     &\leq  \exp(-c\Delta t) W(I x^\ell) + D \Delta t^q+  C (Kh)^2 \Delta t \\
                     &\leq  \exp(-c\Delta t) r_0 + D \Delta t^q+   C (Kh)^2 \Delta t \nonumber\\
                     &\leq r_0 \nonumber
  \end{align}
  whenever
  \begin{equation}\label{eq: second condition on h0}
   D \tmax^q +  C(Kh_0)^2 \Delta t \leq (1-\exp(-c \Delta t))\frac{r_0}{2}.
  \end{equation}
  This second condition on $h_0$ and $\tmax$ is stronger than the first~\eqref{eq: first condition on h0}, so we conclude that $Ix^{\ell+1}  \in \mathscr{M}$ if~\eqref{eq: second condition on h0} holds.

  Now suppose that $Ix^\ell \in \mathscr{M}$ and $m(x^\ell) \leq h$. Suppose that the first reparametrization after step $\ell$ occurs at step $\ell +r+1$ for some $r \geq 0$.  By the the previous paragraph, $I x^{\ell+q} \in \mathscr{M}$  for $q = 1, \dots, r$. Therefore, inequality~\eqref{eq: lyapunov inequality, no reparametrization} yields
  \begin{equation*}
    W(Ix^{\ell+q}) \leq  \exp(-c\Delta t)  W(Ix^{\ell+q-1}) + D \Delta t^q + C (Kh)^2 \Delta t
  \end{equation*}
  for $q = 1, \dots, r$.
  It follows by~\eqref{eq: Lipschitz property of Lyapunov function} and~\eqref{eq: second condition on h0} that the variation of parameters formula~\eqref{eq: variation of parameters formula} holds up to step $\ell+r$. Using Lemma~\ref{lem: bound on interp time} and this variation of parameters formula, we have
  \begin{align*}
    W(Ix^{\ell+r}) &\leq \exp(-c r \Delta t) W(Ix^\ell) + \frac{(D \Delta t^q+C(Kh)^2) \Delta t(1- \exp(-cr \Delta t))}{1- \exp(-c \Delta t)} \\
                   &\leq \exp(-c \tmin) r_0 + (1- \exp(-c \tmin))\frac{r_0}{2} \\
                   &\leq (1+\exp(-c \tmin))\frac{r_0}{2}.
  \end{align*}

  Reparametrization occurs at the $(\ell+r+1)$'st step. Therefore, $m(S_{\Delta t} x^{\ell+r}) > Kh$. However, since $S_{\Delta t}$ has Lipschitz constant $\exp(L\Delta t)$ and $m(x^{\ell+r}) \leq Kh$,
  \begin{equation*}
    m(S \Delta_t x^{\ell+r}) \leq \exp(L\Delta t) Kh.
  \end{equation*}
  By Assumption~\ref{asm: truncation error}, Lemma~\ref{lem: interp error}, and Lemma~\ref{lem: error in commuting evolution and interpolation}, we then have
  \begin{align*}
    \hd(Ix^{\ell+1},S_{\Delta t} I x^\ell) &= \hd(I R S_{\Delta t}x^\ell,S_{\Delta t} I x^\ell)\\
                                           &\leq \exp(L\Delta t) K h +D \Delta t^q + C (Kh)^2 \Delta t.
  \end{align*}
  Therefore, by arguments similar to those in the previous paragraph, if
  \begin{equation}\label{eq: third condition on h0}
    \exp(L\Delta t) K h_0 +D \Delta t^q+ C(Kh)^2 \Delta t \leq (1 - \exp(-c \tmin)) \frac{r_0}{2},
  \end{equation}
  then $I x^{\ell+r+1} \in \mathscr{W}$ and
  \begin{align*}
    W(Ix^{\ell+r+1})& \leq \exp(-c \Delta t) W(I x^{\ell+r}) +  K'h +D \Delta t^q+ C(Kh)^2 \Delta t \\
                    &\leq (1+\exp(-c \tmin)) \frac{r_0}{2} + (1 - \exp(-c \tmin)) \frac{r_0}{2} \\
                    &= r_0.
  \end{align*}
  Thus, $I x^{\ell+r+1} \in \mathscr{M}$.

  Let $h_0$ and $\tmax$ satisfy~\eqref{eq: second condition on h0} and~\eqref{eq: third condition on h0}. We have seen that if $Ix^\ell \in \mathscr{M}$, $m(x^\ell) \leq h_0$, and $\Delta t \leq \tmax$, then the trajectory of the string method remains in $\mathscr{M}$ until after the first reparametrization, say at step $r$. When the reparametrization occurs, we have $I x^r \in \mathscr{M}$ and $m(x^r) \leq h$. Therefore, by induction, if $Ix^0 \in \mathscr{M}$ and $m(x^0) \leq h$, then $Ix^\ell \in \mathscr{M}$ for all $\ell \geq 0$.
\end{proof}

\section*{Acknowledgement}
We would like to thank Eric {Vanden-Eijnden} for discussions that motivated this work.

\end{document}